\documentclass[11pt]{article}
 \usepackage{amsmath, amsthm, amssymb, bbm, setspace,bigints}
 \usepackage[margin=1 in]{geometry}
\usepackage{caption}
\usepackage{subcaption}
\usepackage[toc,page]{appendix}
\usepackage{cite}         
\usepackage[pdftex]{graphicx}
\usepackage{pdfpages}
\usepackage{epstopdf}
\usepackage{booktabs}
\usepackage{authblk}

\usepackage{amsthm,bbm,amssymb}
\usepackage{amsmath}
\usepackage{natbib}
\usepackage[colorlinks,citecolor=blue,urlcolor=blue,filecolor=blue,backref=page]{hyperref}
\usepackage{graphicx}

\usepackage[utf8]{inputenc}
\usepackage{amsmath,amssymb,natbib,xcolor,multicol,url,mhequ}
\usepackage{graphicx,bbm,xr}
\usepackage{booktabs,epstopdf,color}
\usepackage[space]{grffile}
\usepackage{lineno}
\usepackage{multirow}

\def\m{\mathcal}
\def\mb{\mathbb}

\newcommand{\abs}[1]{\left\vert#1\right\vert}

\newcommand{\norm}[1]{\left\lVert #1 \right\rVert}

\DeclareMathOperator*{\argmin}{arg\,min}

\def\m{\mathcal}
\def\mb{\mathbb}

\newcommand \bbP{\mathbb{P}}
\newcommand \bbE{\mathbb{E}}
\newcommand{\be}{\begin{equs}}
\newcommand{\ee}{\end{equs}}

\numberwithin{equation}{section}
\theoremstyle{plain}
\newtheorem{theorem}{Theorem}[section]
\newtheorem{lem}{Lemma}[section]

\newtheorem{corollary}{Corollary}[section]
\newtheorem{proposition}{Proposition}[section]
\newtheorem{remark}{Remark}[section]

\usepackage{lipsum}
\usepackage{amsfonts,amsmath}
\usepackage{tikz-cd}
\usepackage{epstopdf}
\usepackage{algorithm} 
\usepackage{algorithmic}  
\usepackage[algo2e,ruled,vlined]{algorithm2e}

\title{Statistical Guarantees for Transformation Based Models with Applications to Implicit Variational Inference}

\author[1]{Sean Plummer\thanks{splumme@tamu.edu, contributed equally to this work}}
\author[2]{Shuang Zhou\thanks{szhou98@asu.edu, contributed equally to this work}}
\author[1]{Anirban Bhattacharya\thanks{anirbanb@stat.tamu.edu}}
\author[3]{David Dunson\thanks{dunson@duke.edu}}
\author[1]{Debdeep Pati\thanks{debdeep@stat.tamu.edu}}
\affil[1]{Department of Statistics,  Texas A\&M University, College Station, TX 77843, USA}
\affil[2]{School of Mathematical and Statistical Sciences,  Arizona State University, Tempe, AZ 85281, USA}
\affil[3]{Department of Statistics,  Duke University, Durham, NC 27708, USA}
\begin{document}
\maketitle
\begin{abstract}
Transformation-based methods have been an attractive approach in non-parametric inference for problems such as unconditional and conditional density estimation due to their unique hierarchical structure that models the data as flexible transformation of a set of common latent variables. More recently, transformation-based models have been used in variational inference (VI) to construct flexible implicit families of variational distributions. However, their use in both non-parametric inference and variational inference lacks theoretical justification. We provide theoretical justification for the use of non-linear latent variable models (NL-LVMs) in non-parametric inference by showing that the support of the transformation induced prior in the space of densities is sufficiently large in the $L_1$ sense. We also show that, when a Gaussian process (GP) prior is placed on the transformation function, the posterior 
concentrates at the optimal rate up to a logarithmic factor. Adopting the flexibility demonstrated in the non-parametric setting, we use the NL-LVM to construct an implicit family of variational distributions, deemed GP-IVI. We delineate sufficient conditions under which GP-IVI achieves optimal risk bounds and approximates the true posterior in the sense of the Kullback-Leibler divergence. To the best of our knowledge, this is the first work on providing theoretical guarantees for implicit variational inference.
\end{abstract}
\section{Introduction}

Transformation-based models are a powerful class of latent variable models, which rely on a hierarchical 
generative structure for the data.
In their simplest form, these models have the following structure
\begin{eqnarray}
y_i & = & \mu(x_i) + \epsilon_i,\quad \epsilon_i \sim N(0,\sigma^2), \nonumber \\
x_i & \stackrel{iid}{\sim} & g,\label{eq:base}
\end{eqnarray}
for $i=1,\ldots,n$, where $y_i \in \mb R$ is a real-valued observed variable, $\mu$ is the `transformation' function, $x_i$ is a latent (unobserved) variable underlying $y_i$, $g$ is a known density of the latent data (e.g., uniform or standard normal), and we include a Gaussian measurement error with variance $\sigma^2$.  For simplicity in exposition, we consider a very simple case to start but one can certain include multivariate $x_i$ and $y_i$ and other elaborations.

Model (\ref{eq:base}) and its elaborations include many popular methods in the literature.  If we choose a Gaussian process (GP) prior for the function $\mu$, then we obtain a type of GP Latent Variable Model (GP-LVM) \citep{lawrence2004gaussian, lawrence2005probabilistic,lawrence2007hierarchical}.  We can also obtain kernel mixtures as a special case; for example, by choosing a discrete distribution for $g$.  The extremely popular Variational Auto-Encoder (VAE) is based on choosing a deep neural network for $\mu$, and then obtaining a particular variational approximation relying on a separate encoder and decoder neural network \citep{kingma2013auto}.  Refer also to the non-linear latent variable model  (NL-LVM) framework of  
\citep{kundu2014latent} for a nonparametric Bayesian perspective on models related to (\ref{eq:base}).

However, a key question is what are the theoretical properties of `transformation' based models of the form in (\ref{eq:base}).  For example, can this framework be theoretically used to approximate any density with an arbitrarily high degree of accuracy?  Does the accuracy improve with sample size as the optimal rate for density estimation or conditional density estimation (given fixed covariates) problems?

These types of questions have been answered elegantly for many nonparametric Bayes and frequentist density estimation methods. For example, Dirichlet process mixture models (DPMMs) have been very widely 
applied \citep{ferguson1973bayesian,ferguson1974prior,muller1996bayesian,maceachern1999dependent,escobar1995bayesian} and studied in terms of their optimality properties asymptotically \citep{ghosal1999posterior,ghosal2000convergence,ghosal2007posterior,kruijer2010adaptive}. 

When using a continuous distribution $g$, model \eqref{eq:base} leads to a specific class of continuous transformation-based model as the NL-LVM models. Here a GP prior is a natural choice for the unknown transformation \citep{lenk1988lognormal,lenk1991nonparametric,tokdar2010regression, kundu2014latent,tokdar2007estimation, dasgupta2017geometric}, however this approach has two main setbacks. The primary tools used to develop these theoretical results in the context of DPMMs, approximating an arbitrary smooth density using a convolution against a discrete mixing measure, cannot be extended to NL-LVMs in a straightforward manner. The alternative approach using Markov chain Monte Carlo methods comes with theoretical guarantees, but suffers from computational instability owing to lack of conjugacy. This instability propagates through the posterior distribution of the unknown transformation requiring expert parameter tuning and vigilance for guaranteed performance. To mitigate some of these issues associated with a full-blown MCMC, approximate Bayesian methods including the variational inference are proposed \citep{titsias2010bayesian}.

Development of flexible variational families using the reparametrization trick \citep{kingma2013auto, kingma2015variational, jankowiak2018pathwise, figurnov2018implicit} have emerged as a powerful idea over the last decade and continues to flourish, often in parallel with latest developments in generative deep-learning methods. 
While the overarching goal of this trick is to find unbiased estimates of the gradient of the objective function (evidence lower bound in variational inference), one cannot but notice its connection with non-linear latent variable methods. A similar idea is explored as {\em Implicit variational inference}  \citep{huszar2017variational, shi2017kernel} to construct an implicit distribution, a distribution that cannot be analytically specified but can be sampled from. Such a construction brings in certain computational challenges stemming from density ratio estimation. 
More recently, implicit VI was extended to semi-implicit VI \citep{yin2018semi, molchanov2019doubly,titsias2019unbiased} which avoids density ratio estimation by using a semi-implicit variational distribution $q_\phi(\theta)=\int q\{\theta \mid g_\phi(u)\} q(u)du$ where the density $q\{z\mid g_\phi(u)\}$ corresponds to a transformation-based model with transformation $g_\phi$ -- typically taken to be a neural network with parameters $\phi$.  Although VI approaches have shown significant improvements in computational speed their theoretical properties are largely a mystery.

By developing a framework to answer the two previous theoretical questions for NL-LVM with a continuous transform, we are able to find a novel approach to implicit variational inference based on the NL-LVM, for which we can provide strong theoretical guarantees.
Building off of the framework from \cite{kundu2014latent}, we provide a rate-adaptive result for a class of NL-LVM models for density estimation by assigning a rescaled GP prior on the transformation function, and in the process significantly advance the technical understanding of NL-LVM. We provide conditions for the mixing measure to admit a density with respect to Lebesgue measure and show that the prior support of the NL-LVM is at least as large as that of DPMMs. We use the same class of NL-LVM models to construct a flexible implicit variational family, deemed GP-IVI. We show that Kullback--Leibler (KL) divergence between the GP-IVI and the true posterior is stochastically bounded, which is the best possible attainable bound. Additionally, we show that GP-IVI achieves the optimal variational risk bound. To the best knowledge of the authors, these are the first theoretical results in the context of implicit variational inference methods.




\noindent {\em A summary of our contributions.}
Our results are the first to provide a concrete theoretical framework for transformation-based models widely used in Bayesian inference and machine learning. By establishing a connection between NL-LVM with implicit family of distributions, we provide statistical guarantees for implicit variational inference.  Motivated by our findings, transformation-based models have the potential to provide machine learning with a rich class of implicit variational inference methods that come with strong theoretical guarantees.


We close the section by defining some notations in \S \ref{ssec:notations} used throughout the paper. In \S \ref{sec:nllvm} we present an overview of the NL-LVM model as well as several properties of the model. In section \S \ref{sec:theory1} we discuss our two main results for non-parameteric inference using NL-LVM. In \S \ref{sec:theory2} we introduce GP-IVI. We then show that that the KL divergence between the variational posterior and the true posterior is stochastically bounded and argue why this is optimal from a statistical perspective. Inspired by \cite{yang2020alpha}, we additionally present parameter risk bounds of a version of implicit variational inference, which we term as $\alpha$-GP-IVI which is obtained by raising the likelihood to a fractional power $\alpha \in (0, 1)$.

\subsection{Notation} \label{ssec:notations}
We denote the Lesbesgue measure on $\mathbb{R}^p$ by $\lambda$. The supremum norm and $\mbox{L}_1$-norm are denoted by $\norm{\cdot}_{\infty}$ and $\norm{\cdot}_{1}$, respectively. For two density functions $p, q \in \mathcal{F}$, let $h$ denote the Hellinger distance defined as $h^2(p, q) =\int (p^{1/2} - q^{1/2})^2 d\lambda $. Denote the Kullbeck-Leibler divergence between two probability densities $p$ and $q$ with respect to the Lebesgue measure by $D(p||q)=\int p\log(p\slash q) d\lambda$.  We define the additional discrepancy measure $V(p||q)=\int p \log^2(p\slash q)d\lambda$, which will be referred to as the V-divergence. For a set $A$ we use $I_A$ to denote its indicator function. We denote the density of the normal distribution $N(t;0,\sigma^2I_d)$ by $\phi_\sigma(t)$. We denote the convolution of $f$ and $g$ by $f*g(y) = \int f(y - x)g(x)dx$. Absolute continuity of $q$ with respect to $p$ will be denoted $q\ll p$. We denote the set of all probability densities $f \ll \lambda$ by $\mathcal{F}$. The support of a density $f$ is denoted by supp($f$).  For a set $\mathcal{X}$, let $C(\mathcal{X})$ and $C^\beta(\mathcal{X})$, $\beta>0$ denote the spaces of continuous functions and $\beta$-H\"{o}lder space, respectively. 
We write "$\precsim$" for inequality up to a constant multiple. For any $a>0$ denote $\lfloor a \rfloor$ the largest integer that is no greater than $a$. 

\section{A specific transformation-based model}\label{sec:nllvm}
In this section, we focus on an NL-LVM  model \citep{kundu2014latent} in which the response variables are modeled as unknown functions (referred to as the transfer function) of uniformly distributed latent variables with an additive Gaussian error.
This is clearly a specific instance of a transformation-based method; since the inverse c.d.f. transform of uniform random variables can generate draws from any distribution, a prior with large support on the space of transfer functions can approximate draws from any continuous distribution function arbitrarily closely. One can also conveniently approximate a parametric family with  the non-parametric model  by centering the prior on the transfer function on a parametric class of inverse c.d.f. functions. We start from the model formulation and then present a general approximation result of NL-LVM model to the true density under mild regularity conditions.

 \subsection{The NL-LVM model} 
 Suppose we have IID observations $Y_i\in \mb R$ for $i=1,\ldots,n$ with density $f_0 \in \mathcal{F}$, the set of all densities on $\mathbb{R}$ absolutely continuous with respect to the Lebesgue measure $\lambda$. 
 We consider a non-linear latent variable model
\begin{align}\label{eq:nl-lvm}
Y_i &= \mu(\eta_i) + \epsilon _i, \quad \epsilon_i \sim \mbox{N}(0, \sigma^2), \, i=1, \ldots, n\nonumber\\
\mu &\sim \Pi_{\mu}, \quad \sigma \sim \Pi_{\sigma},\quad \eta_i \sim \mbox{U}(0,1),
\end{align}
where $\eta_i$'s are latent variables, $\mu \in C[0, 1]$ is a \emph{transfer function} relating the latent variables to the observed variables and $\epsilon_i$ is an idiosyncratic error. Marginalizing out the latent variable, we obtain the density of $y$ conditional on the transfer function $\mu$ and scale $\sigma$
\begin{eqnarray}\label{eq:gpt_def}
f(y; \mu, \sigma) \stackrel{\text{def}}{=} f_{\mu, \sigma}(y)= \int_{0}^{1}\phi_{\sigma}(y-\mu(x))dx.
\end{eqnarray}
It is not immediately clear whether the class of densities $\{f_{\mu,\sigma}\}$ encompasses a large subset of the density space. The following intuition relates the above class with
continuous convolutions which plays a key role in our proofs. 
Within the support of a continuous density $f_0$, its cumulative distribution function $F_0$ is strictly monotone and hence has an inverse $F_0^{-1}$ satisfying $F_0\{F_0^{-1}(t)\} = t$ for all $t\in \mbox{supp}(f_0)$. Now letting $\mu_0(x) = F_0^{-1}(x)$, one obtains $f_{\mu_0,\sigma}(y) = \phi_{\sigma}*f_0$, the convolution of $f_0$ with a normal density having mean $0$ and standard deviation $\sigma$. This provides a way to approximate $f_0$ by the NL-LVM, as an important property to bounding the KL-divergence. We summarize the approximation result in the next section.

Let $\tilde{\lambda}$ denote the Lebesgue measure on $[0,1]$ and denote the Borel sigma-field of $\mb R$ by $\m B$. For any measurable function $\mu : [0,1] \to \mathbb{R}$, let $\nu_{\mu}$ denote the induced measure on $(\mathbb{R}, \mathcal{B})$, then, for any Borel measurable set $B$, $\nu_{\mu}(B) = \tilde{\lambda}(\mu^{-1}(B))$. 
By the change of variable theorem for induced measures,
\small
\begin{eqnarray}\label{eq:dpgp}
\int_{0}^{1}\phi_{\sigma}(y-\mu(x))dx = \int \phi_{\sigma}(y-t) d\nu_{\mu}(t),
\end{eqnarray}
\normalsize
so that $f_{\mu, \sigma}$ in \eqref{eq:gpt_def} can be expressed as a kernel mixture form 
with mixing distribution $\nu_{\mu}$. It turns out that this mechanism of creating random distributions is very general. Depending on the choice of $\mu$, one can create a large variety of mixing distributions based on this specification. For example, if $\mu$ is a strictly monotone function, then $\nu_{\mu}$ is absolutely continuous with respect to the Lebesgue measure, while choosing $\mu$ to be a step function, one obtains a discrete mixing distribution.

\subsection{Assumptions on true data density $f_0$}
It is widely recognized that one needs certain smoothness assumptions and tail conditions on the true density $f_0$ to derive posterior convergence rates. We make the following assumptions:

\textbf{Assumption F1} We assume $\log{f_0}\in C^{\beta}[0, 1]$. Let $l_j(x)=d^j/dx^j\{\log f_0(x)\}$ be the $j$th derivative for $j=1, \dots, r$ with $r= \lfloor \beta \rfloor$. For any $\beta >0$, we assume that there exists a constant $L > 0$ such that
\begin{eqnarray}\label{eq:tails}
|l_r(x) - l_r(y)| \le L|x-y|^{\beta - r},
\end{eqnarray}
for all $x \neq y$.

The smoothness assumption in the log scale will be used to obtain an optimal approximation error of the GP-transformation-based model to the true $f_0$, providing a key piece in managing the KL-divergence between the true and the model for posterior inference. Similar assumption on the local smoothness appeared in \cite{kruijer2010adaptive}, while in our case a global smoothness assumption is sufficient since $f_0$ is assumed to be compactly supported.

\textbf{Assumption F2} We assume $f_0$ is compactly supported on $[0,1]$, and that there exists some interval $[a,b] \subset [0,1]$ such that $f_0$ is non-decreasing on $[0, a]$, bounded away from $0$ on $[a,b]$ and non-increasing on $[b, 1]$. 

 Assumption \textbf{F2} guarantees that for every $\delta > 0$, there exists a constant $C > 0$ such that $f_0 * \phi_{\sigma} \geq Cf_0$ for every $\sigma < \delta$. 
 Also see \cite{ghosal1999posterior} for similar assumption in density estimation.

\subsection{Approximation property}
As mentioned above, the flexibility of $f_{\mu,\sigma}$ comes from a large class of the induced density measure $\nu_{\mu}$. Now we discuss the approximation of $f_{\mu,\sigma}$ to the true $f_0$ where we utilize its equivalent form of a convolution with a Gaussian kernel. It is well known that the convolution $\phi_{\sigma}*f_0$ can approximate $f_0$ arbitrary closely as the bandwidth $\sigma \to 0$.  For H\"{o}lder-smooth functions, the order of approximation can be characterized in terms of the smoothness. If $f_0 \in C^{\beta}[0, 1]$ with $\beta \le 2$, the standard Taylor series expansion guarantees that $|| \phi_{\sigma}*f_0 - f_0||_{\infty} = O(\sigma^{\beta})$. However, for $\beta > 2$, it requires higher order kernels for the convolution to remain the optimal error \citep{wand1994kernel,devroye1992note}. \cite{kruijer2010adaptive} proposed an iterative procedure to construct a sequence of function $\{f_j\}_{j\ge 0}$ by 
\begin{eqnarray}\label{eq:construction}
f_{j+1} = f_0 - \triangle_{\sigma} f_j, \quad \triangle_{\sigma}f_j = \phi_{\sigma}*f_j - f_j, \quad j \ge 0.
\end{eqnarray} 
We define $f_\beta = f_j$ with integer $j$ such that $\beta \in (2j, 2j+2]$. Under such construction, for $f_0 \in C^{\beta}[0,1]$ the convolution $\phi_{\sigma} * f_\beta$ preserves the optimal error $O(\sigma^\beta)$ (Lemma 1 in \cite{kruijer2010adaptive}). We state a similar result in the following. \\



\begin{proposition} \label{approximation}
For $f_0 \in C^{\beta}[0,1]$ with $\beta \in (2j, 2j+2]$ satisfying Assumptions \textbf{F1} and \textbf{F2}, for $f_\beta$ defined as from the iterative procedure \eqref{eq:construction} we have 
$$ \Vert \phi_{\sigma}*f_\beta - f_0\Vert_\infty = O(\sigma^\beta),$$
and 
\begin{align}\label{eq:approximation}
\phi_{\sigma}*f_\beta(x) = f_0(x) (1+D(x)O(\sigma^\beta)),
\end{align}
where
$$
D(x) = \sum_{i=1}^{r} c_i {|l_j(x)|}^{\frac{\beta}{i}} + c_{r+1},
$$
for non-negative constants $c_i, i = 1, \dots, r+1$, and for any $x \in[0,1]$.
\end{proposition}

The proof can be found in the supplementary file section \ref{sec:approximation}. The ability to represent the model in terms proportional to true density plays an important role in bounding the KL-divergence between $f_{\mu,\sigma}$ and $f_0$.  

\begin{remark}
The approximation result can be extended to the isotropic $\beta$-H\"{o}lder space $\mathcal{C}^\beta[0,1]^d$ under similar regularity assumptions. The extended approximation result can be applied to more general cases.
\end{remark}

\section{Posterior inference for NL-LVM}\label{sec:theory1}

Most of the existing literature on non-parametric Bayesian approaches to the density estimation problem are centered around DP mixture priors \citep{ferguson1973bayesian, ferguson1974prior}, which are simply transformation-based models with a discrete distribution for the latent variables. On the other hand, the theoretical properties of continuous transformation-based models remain largely unknown.

In this section, we provide theoretical results for posterior inference of the transformation-based model for unconditioned density estimation in the context of NL-LVM. Our results are two-fold: (1) We first show that a large class of transfer function $\mu$ leads to $L_1$ large support of the space of densities induced by the NL-LVM; (2) We obtain the optimal frequentist rate up to a logarithmic factor under standard regularity conditions on the true density using the transformation-based approach with induced GP priors. 

\subsection{$L_1$ large support}

One can induce a prior $\Pi$ on $\mathcal{F}$ via the mapping $f_{\mu,\sigma}$ by placing independent priors
$\Pi_{\mu}$ and $\Pi_{\sigma}$ on $C[0,1]$ and $[0, \infty)$ respectively, as $\Pi = (\Pi_{\mu} \otimes \Pi_{\sigma}) \circ f^{-1}_{\mu,\sigma}$. \cite{kundu2014latent} assumes a Gaussian process prior with squared exponential covariance kernel on $\mu$ and an inverse-gamma prior on $\sigma^2$.  
Given the flexibility of $f_{\mu,\sigma}$ upon the choices of $\mu$, placing a prior on $\mu$ supported on the space of continuous functions $C[0, 1]$ without further restrictions is convenient and Theorem \ref{thm:support} assures us that this specification leads to large $L_1$ support on the space of densities.

Suppose the prior $\Pi_{\mu}$ on $\mu$ has full sup-norm support on $C[0,1]$ so that $\Pi_{\mu}(\Vert \mu - \mu^*\Vert_{\infty} < \epsilon) > 0$ for any $\epsilon > 0$ and $\mu^* \in C[0,1]$, and the prior $\Pi_{\sigma}$ on $\sigma$ has full support on $[0, \infty)$. If $f_0$ is compactly supported, so that the quantile function $\mu_0 \in C[0,1]$, then it can be shown that under mild conditions, the induced prior $\Pi$ assigns positive mass to arbitrarily small $L_1$ neighborhoods of any density $f_0$. We summarize the above discussion in the following theorem, with a proof provided in the section \ref{sec:support} of supplementary file.

\begin{theorem}\label{thm:support}
If $\Pi_{\mu}$ has full sup-norm support on $C[0,1]$ and $\Pi_{\sigma}$ has full support on $[0, \infty)$, then the $L_1$ support of the induced prior $\Pi$ on $\mathcal{F}$ contains all densities $f_0$ which have a finite first moment and are non-zero almost everywhere on their support.
\end{theorem}
\begin{remark}
The conditions of Theorem \ref{thm:support} are satisfied for a wide range of Gaussian process priors on $\mu$ (for example, a GP with a squared exponential or Mat\'{e}rn covariance kernel).
\end{remark}
\begin{remark}
When $f_0$ has full support on $\mathbb{R}$, the quantile function $\mu_0$ is unbounded near $0$ and $1$, so that $\Vert\mu_0\Vert_{\infty} = \infty$. However, $\int_{0}^{1} \abs{\mu_0(t)} dt = \int_{\mathbb{R}} \abs{x} f_0(x) dx$, which implies that $\mu_0$ can be identified as an element of $L_1[0,1]$ if $f_0$ has finite first moment. Since $C[0,1]$ is dense in $L_1[0,1]$, the previous conclusion regarding $L_1$ support can be shown to hold in the non-compact case too.
\end{remark}

\subsection{Posterior contraction results} 
Gaussian process priors have been widely used in non-parametric Bayesian inference as well as machine learning due to their modeling advantages and proper theoretical grounding \citep{van2007bayesian,van2008rates,van2009adaptive}. Considering a Gaussian process as the transformation mapping over the latent variable, the transformation-based model essentially aligns with a Gaussian process latent variable model (GP-LVM) \citep{lawrence2004gaussian,lawrence2005probabilistic,lawrence2007hierarchical,ferris2007wifi}. 
Theoretical work of GP-LVM such as \cite{kundu2014latent} showed a KL large support of the induced prior process, and also showed the posterior consistency to the true density function. However a straightforward description of the space of densities induced by the proposed model is not clear and the posterior contraction rate of the proposed model for finite data is still unknown. 
 
We now present the posterior contraction result for transformation-based model with NL-LVM. To that end, we first review its definition.  Given independent and identically distributed observations $Y^{(n)}=(Y_1, \ldots, Y_n)$ from a true density $f_0$,  a posterior $\Pi_n$ associated with a prior $\Pi$ on $\mathcal{F}$ is said to contract at a rate $\epsilon_n$, if for a distance metric $d_n$ on $\mathcal{F}$, 
\begin{eqnarray}\label{eq:postcon}
\bbE_{f_0}\Pi_n\{ d_n(f, f_0) > M \epsilon_n \mid Y^{(n)}\} \to 0
\end{eqnarray}
for a suitably large integer $M > 0$. Unlike the treatment in discrete mixture models \citep{ghosal2007posterior} where a compactly supported density is approximated with a discrete mixture of normals, the main idea is to first approximate the true density $f_0$ by a proper convolution with $f_\beta$ defined as in \eqref{eq:construction},  
 then allow the GP prior on the transfer function to appropriately concentrate around $\mu_\beta$, the inverse c.d.f. function of the defined $f_\beta$. We first state our choices for the prior distributions $\Pi_{\mu}$ and $\Pi_{\sigma}$.

\textbf{Assumption P1} We assume $\mu$ follows a centered and rescaled Gaussian process denoted by $\mbox{GP}(0, c^A)$, where $A$ denotes the rescaled parameter, and assume $A$ has density $g$ satisfying for $a>0$,
\begin{align*}
C_1a^p\exp{(-D_1a \log^q a)} \le g(a) \le C_2a^p\exp{(-D_2a \log^q a)}.
\end{align*}
\textbf{Assumption P2} We assume $\sigma \sim \mbox{IG}(a_{\sigma}, b_{\sigma})$.
Note that contrary to the usual conjugate choice of an inverse-gamma prior for $\sigma^2$, we have assumed an inverse-gamma prior for $\sigma$.  This enables one to have slightly more prior mass near zero compared to an inverse-gamma prior for $\sigma^2$, leading to the optimal rate of posterior convergence. Refer also to \cite{kruijer2010adaptive} for a similar prior choice for the bandwidth of the kernel in discrete location-scale mixture priors for densities.\\
 \begin{theorem}\label{thm:compact}
If $f_0$ satisfies Assumptions \textbf{F1} and \textbf{F2} and the priors $\Pi_\mu$ and $\Pi_{\sigma}$ are as in
Assumptions \textbf{P1} and \textbf{P2} respectively, the best obtainable rate of posterior convergence relative to Hellinger metric $h$ is
\begin{eqnarray}\label{eq:optrate}
\epsilon_{n} = n^{-\frac{\beta}{2\beta+1}}(\log n)^{t},
\end{eqnarray}
where $t=\beta(2\vee q)/(2\beta +1) +1$.
\end{theorem}
We provide a sketch of the proof below, the full proof is deferred to the supplementary file section \ref{sec:compact}. It suffices to check sufficient conditions (prior thickness, sieve construction, entropy condition) for posterior contraction result in \cite{ghosal2000convergence}. We first verify the prior thickness. From Lemma 8 of \cite{ghosal2007posterior}, one has
$$
\int f_0 \log \bigg(\frac{f_0}{f_{\mu, \sigma}}\bigg)^{i} \leq h^2(f_0, f_{\mu,\sigma})\bigg(1 + \log \bigg\|\frac{f_0}{f_{\mu, \sigma}}\bigg\|_{\infty}\bigg)^{i},
$$
for $i=1,2$. By Lemma \ref{lem:logsup}, we have $ \log \|f_0/f_{\mu, \sigma}\|_{\infty} \le \Vert \mu -\mu_\beta\Vert_{\infty}/\sigma^2$, and by Lemma \ref{lem:hellinger} and Lemma \ref{lem:KL}, we bound $h^2(f_0, f_{\mu,\sigma})\precsim \Vert \mu-\mu_\beta\Vert_\infty/\sigma^2 + O(\sigma^{2\beta})$. Then we have 
\begin{align*}
\big\{ \sigma \in [\sigma_n, 2\sigma_n], &\Vert \mu - \mu_\beta\Vert_{\infty} \precsim \sigma_n^{\beta+1} \big\} \subset  \{D(f_0 ||f_{\mu, \sigma}) \precsim \sigma_n^{2\beta}, V(f_0 ||f_{\mu, \sigma})\precsim \sigma_n^{2\beta}\}.
\end{align*}
Under assumptions \textbf{P1} and \textbf{P2} the prior thickness is  guarantee by upper bounding $\Pi\big\{ \sigma \in [\sigma_n, 2\sigma_n], \Vert \mu - \mu_\beta\Vert_{\infty} \precsim \sigma_n^{\beta+1} \big\}$. We construct the sieve
\begin{eqnarray*}
\mathcal{F}_n = \{f_{\mu, \sigma}: \mu \in B_n, l_n < \sigma < h_n \}.
\end{eqnarray*}
where $B_n$ denotes the sieve for a GP prior on $\mu$ as defined in \cite{van2009adaptive}. Further we calculate the entropy $N(\bar{\epsilon}_n, \mathcal{F}_n, \norm{\cdot}_1)$ by observing that for $\sigma_2 >\sigma_1 > \sigma_2/2$, 
\small
\begin{align*}
    \norm{f_{\mu_1, \sigma_1} - f_{\mu_2, \sigma_2}}_1 \leq \bigg(\frac{2}{\pi}\bigg)^{1/2}\frac{\Vert\mu_1 - \mu_2\Vert_{\infty}}{\sigma_1} + \frac{3(\sigma_2 - \sigma_1)}{\sigma_1}.
\end{align*}
\normalsize
The entropy bound is obtained applying Lemma \ref{lem:entropy}. Finally, the sieve compliment condition is easily verified by combining the results on GP priors in \cite{van2009adaptive} and tail properties of inverse-gamma distribution of $\sigma$.

%

\section{Gaussian Process Implicit Variational Inference}\label{sec:theory2}
Motivated by the flexibility we have demonstrated for transformation-based models in the non-parametric setting, we construct a flexible implicit variational family of distributions, deemed Gaussian process implicit variational inference (GP-IVI). We provide sufficient conditions under which GP-IVI achieves optimal risk bounds and approximates the true posterior in the sense of the Kullback--Leibler divergence. We begin by defining common terminology used throughout the section and defining GP-IVI.
\subsection{Preliminaries}
We consider IID observations $Y_i \in \mathbb{R}^p$, 
for $i=1,\ldots,n$. Let $P_\theta^{(n)}$ be the distribution of the observations with parameter $\theta\in \Theta \subset \mathbb{R}^d$ that admits a density $p_\theta^{(n)}$ relative to the Lebesgue measure. Let $P_\theta$ denote the prior distribution of $\theta$ that admits a density $p_\theta$ over $\Theta$. With a slight abuse of notation, we will use $p(Y^{(n)}\mid \theta)$ to denote $\bbP_{\theta}^{(n)}$ and its density function. We adopt a frequentist framework and assume a true data generating distribution $\bbP_{\theta^*}^{(n)}$ and a true parameter $\theta^*$. Denote the negative log prior $U(\theta)=-\log p_\theta(\theta)$ and the log-likelihood ratio of $Y_i$, for $i=1,\ldots,n$, by 
\begin{align}
    \ell_i(\theta,\theta^*)=\log[ p(Y_i\mid \theta) \slash p(Y_i\mid \theta^*)]. 
\end{align}
We denote the first two moments of the log-likelihood by 
\small
\begin{align}
    D(\theta^*||\theta)=-E_{\theta^*}^{(n)}[\ell_1(\theta,\theta^*)],\,    \mu_2(\theta^*||\theta)=E_{\theta^*}^{(n)}[\ell_1(\theta,\theta^*)^2]. 
\end{align} 
\normalsize
Lastly denote the appropriate neighborhood around the true parameter $\theta^*$, 
\small
\begin{align}
    B_n(\theta^*,\varepsilon) = \{ \theta \mid  D[p(Y^{(n)}\mid \theta^*)\| p(Y^{(n)}\mid \theta)]\leq n\varepsilon^2, V[p(Y^{(n)}\mid \theta^*) \| p(Y^{(n)}\mid \theta)]\leq n\varepsilon^2\}. \label{kl.ball}
\end{align}
\normalsize
%
\subsection{Gaussian Process Implicit Variational Inference}

Using the NL-LVM model, we can define the  variational family of $\theta$ conditioned on the latent variable $\eta$, with parameters $\mu\in C[0,1]$ and $ \sigma\in (0,\infty)$,
\begin{align*}
    q_{\mu,\sigma}(\theta_i \mid \eta_i) &= \phi_\sigma(\theta_i-\mu(\eta_i)) \\
    \eta_i &\sim U(0,1),\, i=1,\ldots, d.
\end{align*}
Marginalizing over the latent $\eta$ gives us the implict variational distribution,
\begin{align*}
   q_{\mu,\sigma}(\theta) = \int_0^1 \phi_\sigma( \theta -\mu(\eta) ) d\eta.
\end{align*}
 Together this defines the Gaussian process implict variational inference (GP-IVI) family,
 \small
\begin{align*}
    \m Q_{GP}=\left\{q_{\mu,\sigma}(\theta)=\int_0^1 \phi_\sigma( \theta -\mu(\eta) ) d\eta\mid \mu \in C[0,1],\, \sigma > 0 \right\}. 
\end{align*}
\normalsize

\subsection{Approximation Quality of GP-IVI}

In this section, we show that KL divergence between the true posterior and its optimal GP-IVI approximation   is $O_p(1)$. Using a simple example, we show that without further assumptions this bound cannot be improved. We begin the section with said example. 

Consider the following one-dimensional Gaussian-Gaussian Bayesian model for inference of an unknown true mean $\theta^*$ using the model
\begin{align*}
    Y_1,\ldots, Y_n \sim N(\theta,\sigma^2), \quad  \theta \sim N(\mu_0,\sigma_0^2)
\end{align*}
in which $\mu_0,\sigma_0,\sigma$ are all known. Let $\overline{Y}_n, \mu_n, \sigma_n^2$ denote the sample mean, the posterior mean, and variance, respectively. Straight forward calculations show 
\begin{align*}
    &D\left[ N(\theta^*, n^{-1}\sigma^2) || N(\mu_n, \sigma^2_n)\right]\to \chi^2_1,\,\text{weakly}. 
\end{align*}
Even in the simple case of a normal-normal model, we see that the KL divergence between the true data generating distribution and the true posterior does not converge weakly to 0 but instead converges weakly to a stochastically bounded random variable.   

The $O_p(1)$ bound is achieved over a rather small subfamily of GP-IVI. Define the restricted Gaussian family
\begin{align*}
    \Gamma_n=\{N(\mu,\tau^2 I_d)\mid \Vert \mu \Vert_2 \leq M,\, 0\leq \sigma_n\leq \tau \leq c_0^{1\slash 2}\sigma_n\},
\end{align*}
and let $\mu_f$ denote the quantile function corresponding to $f\in \Gamma_n$. We define the corresponding small bandwidth convolution Gaussian (variational) family
\small
\begin{align*}
    \m Q_n=\left \{ q_{\mu,\sigma}(\theta) \mid q_{\mu,\sigma}(\theta)=\int_0^1 \phi_\sigma(\theta-\mu_f(\eta))d\eta,\quad f\in \Gamma_n \right\}. 
\end{align*}
\normalsize
The following assumptions are required to show the $O_p(1)$ bound for the KL-divergence. \\
\\
\textbf{Assumption B1}  The true parameter $\theta^*$ satisfies $\Vert \theta^*\Vert_2\leq M$.\\
\\
\textbf{Assumption B2}  The variance bound $\sigma_n$ satisfies $0\leq \sigma_n\leq n^{-1\slash 2} \leq c_0^{1\slash 2}\sigma_n$, for all $n\geq 1$.
\\
\\
\textbf{Assumption B3}  The quantities $D(\theta^*||\theta)$ and $\mu_2(\theta^*||\theta)$ are finite for all $\theta\in \mb R^d$. 
\\
\\
\textbf{Assumption B4}  The matrices of the second derivatives, $D^{(2)}(\theta^*||\theta)$, $\mu_2^{(2)}(\theta^*||\theta)$, $U^{(2)}(\theta)$ exist on $\mathbb{ R}^d$ and satisfy for any $\theta, \theta'\in \mathbb{ R}^d$,
\begin{align*}
    s_{max}\left( D^{(2)}(\theta^*||\theta)- D^{(2)}(\theta^*||\theta')\right) \leq C\Vert \theta-\theta'\Vert_2^{\alpha_1},\\
    s_{max}\left( \mu_2^{(2)}(\theta^*||\theta)- \mu_2^{(2)}(\theta^*||\theta')\right) \leq C\Vert \theta-\theta'\Vert_2^{\alpha_2},\\
    s_{max}\left( U^{(2)}(\theta)- U^{(2)}(\theta')\right) \leq C\Vert \theta-\theta'\Vert_2^{\alpha_3},
\end{align*}
for some $\alpha_1,\alpha_2,\alpha_3 > 0$. Here $s_{max}$ denotes the maximum eigenvalue of the matrix. 
\\
\\
\textbf{Assumption B5}  $D(\theta^*||\theta)\geq C \Vert \theta-\theta^*\Vert_2$. 

Assumption {\bf B1} is needed so that a normal distribution centered at the true parameter is contained in $\Gamma_n$. Assumptions {\bf B2-B4} are technical assumptions needed in order to achieve convergence of certain bounds used in the proof. Assumption {\bf B5} is a standard identifiability condition. 
\begin{theorem}\label{thm:vi_bound}
Under assumptions \textbf{B1} through \textbf{B5} it holds that $m^*_n(\m Q_n)=\min_{q\in \m Q_n} \left\{D[q||p(\cdot \mid Y^n)]\right\}$ is bounded in probability with respect to the data generating distribution $\bbP_{\theta^*}^{(n)}$. Formally, given any $\varepsilon>0$, there exists $M_\varepsilon, N_\varepsilon >0$ such that for $n\geq N_\varepsilon$, we have $ \bbP_{\theta^*}^{(n)}(m^*_n(\m Q_n)> M_\varepsilon) \leq \varepsilon$.

\end{theorem}
Again, we provide a sketch of the proof below and provide a full proof in section \ref{sec:vi_bound} of the supplementary file. Under assumptions \textbf{B1-B2}, $q_n(\theta)=N(\theta; \theta^*,\sigma^2+\sigma_n^2)$ belongs to $\m Q_n$. By definition, $m_n^*(\m Q_n)\leq D[q_n||p(\cdot\mid Y^{(n)})]$. We show $D[q_n||p(\cdot\mid Y^{(n)})]$ is $O_p(1)$ by showing that it is a sum of $O_p(1)$ terms. Letting $\bbE_n$ denote the expectation with respect to $q_n$,  $D[q_n||p(\cdot\mid Y^{(n)})]$ can be broken into four parts $\bbE_n[\log q_n]$, $\log m(Y^{(n)})$, $\bbE_n[U(\theta)]$, and $\bbE_n\left [\sum_{i=1}^n\ell_i(\theta,\theta^*)\right]$. The first term $\bbE_n[\log q_n]$ is a constant, hence $O_p(1)$. Noting $E^{(n)}_{\theta^*}[ m(Y^{(n)})]=1$, an application of Markov's inequality shows that $\log m(Y^{(n)})$ is $O_p(1)$. Taking a (multivariate) Taylor expansion of the functions  $U(\theta)$, $D(\theta^*||\theta)$, and $\mu_2(\theta^*||\theta)$ about $\theta^*$ and applying assumption \textbf{B4} and \textbf{B5} gives us the bounds 
\begin{align}
C_\ell(\sigma^2+\sigma_n^2)\leq \bbE_n[D(\theta^*||\theta)]\leq C_u(\sigma^2+\sigma_n^2), \nonumber \\
\bbE_n[\mu_2(\theta^*||\theta)]\leq C_2(\sigma^2+\sigma_n^2), \label{thm41.ub}\\
\bbE_n[U(\theta)]\leq C_1(\sigma^2+\sigma_n^2). \nonumber 
\end{align}
Markov's inequality shows that $U(\theta)$ is $O_p(1)$. It remains to show $\bbE_n\left [\sum_{i=1}^n\ell_i(\theta,\theta^*)\right]$ is $O_p(1)$. Given $\varepsilon >0$, choose $\delta= \left [C_2c_0\slash (\varepsilon C_\ell)^2 \right]^{1\slash 2}$. Applying Chebychev's and Jensen's inequalities together with (\ref{thm41.ub}) we have,
\small
\begin{align*}
    \mb P_{\theta^*}^{(n)}\left\{E_0\left[\sum_{i=1}^n\ell_i(\theta,\theta^*)\right ]  \leq -C_u(1+\delta)n(\sigma^2+ \sigma_n^2) \right\}
\leq \frac{\bbE_n[ \mu_2(\theta^*||\theta)] }{\delta^2 n \left(\bbE_n[ D(\theta^*||\theta)] \right)^2} \leq \frac{C_2}{ C_\ell \delta^2n\sigma_n^2 }.
\end{align*}
\normalsize
Finally by assumption \textbf{B2} we have $c_0 n\leq \sigma_n^{-2}$. Thus
\small
\begin{align*}
    \mb P_{\theta^*}^{(n)}\Bigg\{E_0\left[\sum_{i=1}^n\ell_i(\theta,\theta^*)\right]  \leq -2C_u\left(1+\left [C_2c_0\slash (\varepsilon C_\ell)^2 \right]^{1\slash 2}\right) \Bigg\}
    \leq \varepsilon,
\end{align*}
\normalsize
which shows $\bbE_n\left[\sum_{i=1}^n\ell_i(\theta,\theta^*)\right]$ is $O_p(1)$. Combining the four bounds completes the proof.

\subsection{$\alpha$-Variational Bayes Risk Bound for GP-IVI}
In developing risk bounds for parameter estimation, we use a slight variation of the standard variational objective function for technical simplicity.  
$\alpha$-variational Bayes ($\alpha$-VB) \citep{yang2020alpha} is a variational inference framework that aims to minimize the KL divergence between the variational density and the $\alpha$-fractional posterior \citep{bhattacharya2019fractional}, defined as
$$
    P_\alpha(\theta\in B \mid Y^{(n)})= \frac{\int_B [p(Y^{(n)}\mid \theta)]^\alpha p_\theta(\theta)d\theta}{\int_\Theta [p(Y^{(n)}\mid \theta)]^\alpha p_\theta(\theta)d\theta}.
$$
This leads to the following $\alpha$-VB objective 
\small
\begin{align}
    \widehat{q}(\theta)&=\argmin_{q\in \m Q} D(q||p_\alpha(\cdot\mid Y^{(n)}))=\argmin_{q} \alpha\Psi(q),  \label{avb.obj} 
\end{align}
\normalsize
where 
\small
\begin{align*}
    \Psi(q)= \int_\Theta q(\theta) \log\left[ \frac{p(Y^{(n)}\mid \theta^*)}{p(Y^{(n)}\mid\theta)} \right]d\theta -\alpha^{-1} D[q || p_\theta]. 
\end{align*}
\normalsize
The variational expected log-likelihood ratio will be hence referred to as the model-fit term and the remaining KL term will be hence referred to as the regularization term. 

The importance of the $\alpha$-VB framework comes from its ability to upper bound the variational Bayesian risk, the integral of  $r(\theta,\theta^*)=n^{-1}D_\alpha[p_\theta^{(n)}||p_{\theta^*}^{(n)}]$ with respect to $\widehat{q}(\theta)$, by the variational objective $\Psi(q)$. Minimizing the variational objective in turn minimizes the variational risk. 

Before proceeding we motivate the form of our optimal risk bound. Consider preforming VI over the unrestricted class of densities over $\Theta$. Minimizing the $\alpha$-VB risk bound is achieved by balancing the two terms in terms in $\Psi(q)$. By choosing 
\begin{align*}
    q(\theta)=\frac{p_\theta(\theta) I_{B_n(\theta^*,\varepsilon)}(\theta)}{\bbP_{\theta}\left[B_n(\theta^*,\varepsilon)\right]},
\end{align*}

where $B_n(\theta^*,\varepsilon)$ is defined in (\ref{kl.ball}), the  model-fit term can be shown to be of order $O_p(n\varepsilon^2)$ and the regularization term can be shown to be $\alpha^{-1}\log [ P_\theta\{ B_n(\theta^*,\varepsilon)\}^{-1} ]$, a multiple of the local Bayesian complexity. This is the optimal risk bound for variational inference considering the class of all distributions as the variational family \citep{yang2020alpha}. 
We summarize this in the theorem below. 
\begin{theorem}\label{thm:riskbound}
Assume $\widehat{q}_{\mu,\sigma}$ satisfies \eqref{avb.obj} and $\widehat{q}_{\mu,\sigma} \ll p_\theta$. It holds with $ P_{\theta^*}^{(n)}$-probability at least $1-2\slash[(D-1)^2n\varepsilon^2]$ that,
\small
\begin{align*}
 \int \frac{1}{n}D_{\alpha}^{(n)}(\theta,\theta^*)\widehat{q}_{\mu,\sigma}(\theta)d\theta 
 \leq \frac{D\alpha  }{1-\alpha}\varepsilon^2 + \frac{1}{n(1-\alpha)}\log\left\{ P_\theta\left[ B_n(\theta^*,\varepsilon)\right]^{-1}\right\} + O(n^{-1}). 
\end{align*}
\normalsize
\end{theorem}
We provide a sketch of the proof below. The full proof can be found in section \ref{sec:riskbound} of the supplementary file. Following our above motivation, we aim to show that there is a member of the GP-IVI family $\m Q_{GP}$ such that the model-fit term is of order $O_p(n\varepsilon^2)$ and the regularization term is proportional to the local Bayesian complexity. We leverage the approximation properties from \S \ref{sec:theory1} to construct an approximation that achieve this balance. We construct this variational distribution as follows. 

Let the prior distribution of $\theta$ is given by the density $p_{\theta}(\theta)=f_0(\theta)\in C^\beta[0,1]$, $\beta\in (2j,2j+2]$. Let $f_\beta=f_j$ be the density constructed as in (\ref{eq:construction}) satisfying $\Vert\phi_\sigma * f_\beta -f_0\Vert_\infty=O(\sigma^\beta)$. Define the density function
\small
\begin{align}
    \widetilde{f}_\beta(t)=\frac{f_\beta(t)I_{B_n(\theta^*,\varepsilon)}}{\int_{B_n(\theta^*,\varepsilon)}f_\beta(t)dt}
\end{align}
\normalsize
and its corresponding variational density 
\small
\begin{align}
    q_{\widetilde{f}_\beta,\sigma}(\theta)=\int_{-\infty}^{\infty}\phi_\sigma(\theta-t)\widetilde{f}_\beta(t) dt.
\end{align}
\normalsize
The model-fit term is bounded in high probability using a straight forward application of Chebychev's inequality. Using \eqref{eq:approximation}, we  bound the regularization term proportional to the local Bayesian complexity. Combining these and using Theorem 3.2 of \cite{yang2020alpha} finishes the proof. \\
\textbf{Assumption A1} Prior density $p_\theta$ satisfies $\log[P_\theta\{B_n(\theta^*,\varepsilon)\}^{-1}]\leq -n\varepsilon^2$.
\begin{remark}
Let 	$\{p_\theta, \theta \in \Theta\}$ be a parametric family of densities. Assume for 
$\theta, \theta_1, \theta_2$, there exists $\alpha > 0$ such that $D(\theta^* \| \theta) \precsim \|\theta^* - \theta\|^{2\alpha}$, 
$\mu_2(\theta^* \| \theta) \precsim \|\theta^* - \theta\|^{2\alpha}$, and $\|\theta_1 - \theta_2 \|^{\alpha}\precsim h(\theta_1,\theta_2)\precsim \|\theta_1 - \theta_2 \|^{\alpha}$. Then if the prior measure possesses a density that is uniformly bounded away from zero and infinity on $\Theta$, then Assumption
{\bf A1} is satisfied.  Assumptions of this form are  common in the literature; refer to  pg 517 \citep{ghosal2000}.
\end{remark}
\begin{corollary}
Suppose the prior density $p_\theta$ satisfies Assumption {\bf A1} and  $\widehat{q}$ satisfies \eqref{avb.obj}. It holds with probability tending to one as $n\to \infty$ that,
\begin{align*}
    \left\{\int h^2[p(\cdot\mid \theta)|| p(\cdot\mid \theta^*)] \widehat{q}_{\mu,\sigma}(\theta)d\theta\right\}^{1\slash 2} \leq  O(n^{-1}),
\end{align*}
demonstrating that the risk bound is parametric even when a flexible class of variational approximation is used. 
\end{corollary}
\section{Conclusion}
To summarize, we have provided theoretical properties of transformation-based model in non-parametric and variational inferences in the context of NL-LVM. We characterized the space of densities induced by NL-LVM as kernel convolutions with a general class of continuous mixing measures and showed $L_1$ prior support of the transformation-based model. Placing
a GP prior on the transformation function, we obtain the optimal rate of posterior contraction up to a logarithmic factor. Adopting the flexibility of GP-LVM we constructed GP-IVI. We have shown that GP-IVI achieves both optimal $\alpha$-VI risk bounds and optimal approximation to the true posterior. In doing so, we have provided theoretical guarantees for a novel transformation-based implicit variational inference.    


\appendix

\section{Proofs of results in the main document} \label{sec:main}
\subsection{Conventions}
Equations in the main document are cited as (1), (2) etc., retaining their numbers, while new equations defined in this document are numbered (S1), (S2) etc.  In this section we collect the proof of Proposition \ref{approximation}, Theorems \ref{thm:support}, \ref{thm:compact}, \ref{thm:vi_bound} and \ref{thm:riskbound}.  

\subsection{Proof of Proposition \ref{approximation}}\label{sec:approximation}
In this section we prove the results in Proposition \ref{approximation}.\\

\textbf{Proposition 2.1 }For $f_0 \in C^{\beta}[0,1]$ with $\beta \in (2j, 2j+2]$ satisfying Assumptions \textbf{F1} and \textbf{F2}, for $f_\beta$ defined as from the iterative procedure \eqref{eq:construction} we have 
\begin{eqnarray*}
\Vert \phi_{\sigma}*f_\beta - f_0\Vert_\infty = O(\sigma^\beta),
\end{eqnarray*}
and 
\begin{align}\label{eq:approximation}
\phi_{\sigma}*f_\beta(x) = f_0(x) (1+D(x)O(\sigma^\beta)),
\end{align}

where
\begin{eqnarray*}
D(x) = \sum_{i=1}^{r} c_i {|l_j(x)|}^{\frac{\beta}{i}} + c_{r+1},
\end{eqnarray*}
for non-negative constants $c_i, i = 1, \dots, r+1$, and for any $x \in[0,1]$.

\textit{Proof. }
We now show equation \eqref{eq:approximation}. Following the proof of Lemma 1 in \cite{kruijer2010adaptive}, for any $x, y \in [0,1]$,
\begin{eqnarray*}
\log{f_0(y)} \le \log{f_0(x)} + \sum_{i=1}^{r} \frac{l_j(x)}{j!}(y-x)^j + L|y-x|^{\beta},\\
\log{f_0(y)} \ge \log{f_0(x)} + \sum_{i=1}^{r} \frac{l_j(x)}{j!}(y-x)^j - L|y-x|^{\beta}.
\end{eqnarray*}
Define
\begin{eqnarray*}
B^u_{f_0,r}(x,y) = \sum_{i=1}^{r} \frac{l_j(x)}{j!}(y-x)^j + L|y-x|^{\beta},\\
B^l_{f_0,r}(x,y) = \sum_{i=1}^{r} \frac{l_j(x)}{j!}(y-x)^j - L|y-x|^{\beta}.
\end{eqnarray*}
Then we have
\begin{eqnarray*}
e^{B^u_{f_0,r}} \le 1 + B^u_{f_0,r} + \frac{1}{2!}(B^u_{f_0,r})^2 + \dots + M |B^u_{f_0,r}|^{r+1},\\
e^{B^l_{f_0,r}} \ge 1 + B^l_{f_0,r} + \frac{1}{2!}(B^l_{f_0,r})^2 + \dots - M |B^l_{f_0,r}|^{r+1}.
\end{eqnarray*}

where 
\begin{eqnarray*}
M = \frac{1}{(r+1)!} \exp \bigg \{\sup_{x, y  \in [0, 1], x \neq y} \bigg(\bigg|\sum_{j=1}^{r} \frac{l_j(x)}{j!} (y-x)^j\bigg| + L|y-x|^{\beta}\bigg)\bigg \}.
\end{eqnarray*}

Note that $f_0$ is bounded on $[0, 1]$, we consider the convolution on the whole real line by extending $f_0$ analytically outside $[0, 1]$. For $\beta \in (1, 2], r = 1$ and $x \in (0,1)$, 
\small
\begin{align}\label{eq:integral1}
&\phi_{\sigma}* f_0(x)\le f_0(x) \int e^{B^u_{f_0,r}(x, y)} \phi_{\sigma}(y-x) dy\nonumber\\
&\le f_0(x) \int_{\mathbb{R}} \phi_{\sigma}(y-x) [ 1 + L|y-x|^{\beta} +M \{ l^2_1(x)(y-x)^2 + 2Ll_1(x)(y-x)|y-x|^{\beta} + L^2|y-x|^{2\beta} \}] dy.
\end{align}
\normalsize
Since $l_j(x)$'s are all continuous on $[0, 1]$, there exist finite constants $M_j$ such that $|l_j| \le M_j$ and $|y-x| \le 1$. The integral in the last inequality in \eqref{eq:integral1} can be bounded by 
\begin{eqnarray*}
\int_{\mathbb{R}} \phi_{\sigma}(y-x) [ 1 + L|y-x|^{\beta} + M  \{M_1^{2-\beta} |l_1(x)(y-x)|^{\beta} + (L^2 +2M_1)|y-x|^{\beta} \}] dy
\end{eqnarray*}
Therefore,
\begin{eqnarray*}
\phi_{\sigma}* f_0(x) \le f_0(x) \{ 1 + (r_1|l_1(x)|^{\beta} +r_2) \sigma^{\beta}\},
\end{eqnarray*}
where $ r_1 = M M_1^{2-\beta}\mu'_{\beta}, \  \ r_2 =  L(1 + ML + 2MM_1)\mu'_{\beta}$, and $\mu'_{\beta} = \bbE\{|y-x|^{\beta}\}$. \\
In the other direction, 
\small
\begin{eqnarray*}
\phi_{\sigma}* f_0(x) \ge  f_0(x) \int \phi_{\sigma}(y-x)[\{ 1 - L|y-x|^{\beta} - M \{ l^2_1(x)(y-x)^2 - 2Ll_1(x)(y-x)|y-x|^{\beta} + L^2|y-x|^{2\beta}\}]  dy.
\end{eqnarray*}
\normalsize
Thus we achieve expression of $\phi_\sigma * f_{\beta}$ in Proposition \ref{approximation}.

For any $\beta > 2$ and the integer $j$ such that $\beta \in (2j, 2j +2]$. We define $\phi^{(i)}*f$ as the $i$-folded convolution of $\phi$ with $f$ for any integer $i\ge 1$. First we calculate $\phi_{\sigma}* f_0(x)$, $\phi^{(2)}_{\sigma}* f_0(x)$, $\dots$, $\phi^{(j)}_{\sigma}* f_0(x)$, and by Lemma \ref{lem:approx} we get $\phi_{\sigma}* f_j(x)$. The calculation of $\phi^{(i)}_{\sigma}* f_0(x)$ is the same as that of $\phi_{\sigma}* f_0(x)$ except taking the convolution with $\phi_{\sqrt{i}\sigma}$. The terms $\sigma^2$, $\sigma^4$, $\dots$, $\sigma^{2j}$ caused by the factors containing $|y-x|^k$ for $k<\beta$ in $\phi^{(i)}_{\sigma}* f_0$ can be canceled out by Lemma \ref{lem:approx}.
For terms containing $|y-x|^k$ for $k \ge \beta$, we take out $|y-x|^{\beta}$ and bound the rest by a certain power of $|l_j(x)|$ or some constant. Following an induction in \cite{kruijer2010adaptive}, we can guarantee the approximation error of $\phi_{\sigma}*f_\beta$ is at the order of $O(\sigma^\beta)$. \qed

\subsection{Proof of Theorem \ref{thm:support}} \label{sec:support}

\textbf{Theorem 3.1. }If $\Pi_{\mu}$ has full sup-norm support on $C[0,1]$ and $\Pi_{\sigma}$ has full support on $[0, \infty)$, then the $L_1$ support of the induced prior $\Pi$ on $\mathcal{F}$ contains all densities $f_0$ which have a finite first moment and are non-zero almost everywhere on their support.
\begin{proof}
Let $f_0$ be a density with quantile function $\mu_0$ that satisfies the conditions of Theorem \ref{thm:support}. Observe that $\norm{\mu_0}_1 = \int_{t=0}^1 \abs{\mu_0(t)} dt = \int_{-\infty}^{\infty} \abs{z} f_0(z) dz < \infty$ since $f_0$ has a finite first moment, and thus $\mu_0 \in \mbox{L}_1[0, 1]$. Fix $\epsilon > 0$. We want to show that $\Pi\{ B_{\epsilon}(f_0) \} > 0$, where $B_{\epsilon}(f_0) = \{f ~:~ \norm{f - f_0}_1 < \epsilon\}$.

Note that $\mu_0 \notin C[0, 1]$, so that $\bbP( \norm{\mu - \mu_0}_{\infty} < \epsilon)$ can be zero for small enough $\epsilon$. The main idea is to find a continuous function $\widetilde{\mu}_0$ close to $\mu_0$ in $L_1$ norm and exploit the fact that the prior on $\mu$ places positive mass to arbitrary sup-norm neighborhoods of $\widetilde{\mu}_0$. The details are provided below.

Since $\norm{\phi_{\sigma}*f_0 - f_0}_1 \to 0$ as $\sigma \to 0$, find $\sigma_1$ such that $\norm{\phi_{\sigma}*f_0 - f_0}_1 < \epsilon/2$ for $\sigma < \sigma_1$. Pick any $\sigma_0 < \sigma_1$.
Since $C[0, 1]$ is dense in $\mbox{L}_1[0, 1]$, for any $\delta > 0$, we can find a continuous function $\widetilde{\mu}_0$ such that $\norm{\mu_0 - \widetilde{\mu}_0}_1 < \delta$. Now, $\norm{ f_{\mu, \sigma} - f_{\widetilde{\mu}_0, \sigma} }_1 \leq C \norm{ \mu - \widetilde{\mu}_0 }_1/\sigma$ for a global constant $C$. Thus, for $\delta = \epsilon\,\sigma_0/4$,
\begin{align*}
\big\{f_{\mu, \sigma} ~:~ \sigma_0 < &\sigma < \sigma_1, \norm{\mu - \widetilde{\mu}_0}_{\infty} < \delta \big\} \subset \big\{f_{\mu, \sigma} ~:~ \norm{f_0 - f_{\mu, \sigma}}_1 < \epsilon \big\},
\end{align*}
since $\norm{f_0 - f_{\mu, \sigma}}_1 < \norm{f_0 - f_{\mu_0, \sigma}}_1 + \norm{f_{\mu_0, \sigma} - f_{\widetilde{\mu}_0, \sigma}}_1 + \norm{f_{\widetilde{\mu}_0, \sigma} - f_{\mu, \sigma}}_1$ and $f_{\mu_0, \sigma} = \phi_{\sigma}*f_0$. Thus, $\Pi\{ B_{\epsilon}(f_0) \} > \Pi_{\mu}(\norm{\mu - \widetilde{\mu}_0}_{\infty} < \delta) \, \Pi_{\sigma}(\sigma_0 < \sigma < \sigma_1) > 0$, since $\Pi_{\mu}$ has full sup-norm support and $\Pi_{\sigma}$ has full support on $[0, \infty)$.
\end{proof}

\subsection{Proof of Theorem \ref{thm:compact}}\label{sec:compact}
In this section we will give a detailed proof for the adaptive posterior contraction rate result for the NL-LVM models. \\
\\
\textbf{Theorem 3.2. }If $f_0$ satisfies Assumptions \textbf{F1} and \textbf{F2} and the priors $\Pi_\mu$ and $\Pi_{\sigma}$ are as in
Assumptions \textbf{P1} and \textbf{P2} respectively, the best obtainable rate of posterior convergence relative to Hellinger metric $h$ is
\begin{eqnarray}\label{eq:optrate}
\epsilon_{n} = n^{-\frac{\beta}{2\beta+1}}(\log n)^{t},
\end{eqnarray}
where $t=\beta(2\vee q)/(2\beta +1) +1$.
\begin{proof}
Following \cite{ghosal2000convergence}, to obtain the posterior convergence rate we need to find sequences $\bar{\epsilon}_n,\widetilde{\epsilon}_n \to 0$ with
$n\min\{\bar{\epsilon}_n^2,\widetilde{\epsilon}_n^2\} \to \infty$  such that there exist constants $C_1, C_2, C_3, C_4> 0$ and sets $\mathcal{F}_n \subset \mathcal{F}$ so that,
\begin{align}
& \log N(\bar{\epsilon}_n, \mathcal{F}_n, d) \leq C_1n\bar{\epsilon}_n^2, \label{eq1}\\
& \Pi(\mathcal{F}_n^c) \leq C_3\exp\{-n\widetilde{\epsilon}_n^2(C_2+4)\},\label{eq2} \\
& \Pi\bigg( f_{\mu, \sigma}:  \int f_0 \log \frac{f_0}{f_{\mu, \sigma}} \leq\widetilde{\epsilon}_n^2,\, \int f_0 \log \bigg(\frac{f_0}{f_{\mu, \sigma}}\bigg)^2 \leq \widetilde{\epsilon}_n^2 \bigg) \geq C_4\exp\{-C_2n\widetilde{\epsilon}_n^2\}. \label{eq3}
\end{align}
Then we can conclude that for $\epsilon_n = \max\{\bar{\epsilon}_n, \widetilde{\epsilon}_n\}$ and sufficiently large $M > 0$, the posterior probability
\begin{eqnarray*}
\Pi_n(f_{\mu, \sigma}: d(f_{\mu, \sigma}, f_0) > M\epsilon_n | Y_1, \ldots, Y_n) \to 0 \, \, \text{a.s.}\, P_{f_0},
\end{eqnarray*}

where $P_{f_0}$ denotes the true probability measure whose the Radon-Nikodym density is $f_0$.  To proceed, we consider the Gaussian process $\mu \sim W^A$ given $A$, with $A$ satisfying Assumption \textbf{P1}.

We will first verify (\ref{eq3}) along the lines of \cite{ghosal2007posterior}. Recall $f_\beta$ is defined as from \eqref{eq:construction}, by Lemma \ref{lem:density} we guarantee that $f_\beta$ is a well-defined density. Denote by $\mu_\beta = F^{-1}_{\beta}$ the quantile function of $f_\beta$, then we have $f_{\mu_\beta, \sigma} = \phi_{\sigma}*f_\beta$. Note that
\begin{eqnarray}\label{eq:H2}
h^{2}(f_0, f_{\mu, \sigma}) \precsim h^{2}(f_0, f_{\mu_\beta, \sigma}) + h^{2}(f_{\mu_\beta, \sigma}, f_{\mu, \sigma}).
\end{eqnarray}
Under Assumptions \textbf{F1} and \textbf{F2} and by Lemma \ref{lem:KL}, one obtains
\begin{eqnarray}\label{eq:H}
 h^{2}(f_0, f_{\mu_\beta, \sigma}) \le \int f_0 \log\bigg(\frac{f_0}{f_{\mu_\beta, \sigma}}\bigg) \precsim O(\sigma^{2\beta}).
\end{eqnarray}

From Lemma \ref{lem:hellinger} and the following remark, we obtain
\begin{eqnarray}
h^{2}(f_{\mu_\beta, \sigma}, f_{\mu, \sigma}) \precsim \frac{\norm{\mu- \mu_\beta}_{\infty}^2}{\sigma^2}.
\end{eqnarray}
From Lemma 8 of \cite{ghosal2007posterior}, one has
\begin{eqnarray}\label{eq:KtoH}
\int f_0 \log \bigg(\frac{f_0}{f_{\mu, \sigma}}\bigg)^{i} \leq h^2(f_0, f_{\mu,\sigma})\bigg(1 + \log \bigg\|\frac{f_0}{f_{\mu, \sigma}}\bigg\|_{\infty}\bigg)^{i},
\end{eqnarray}
for $i=1,2$.

From (\ref{eq:H2})-(\ref{eq:KtoH}), for any $b \geq 1$ and $\widetilde{\epsilon}_n^2 = \sigma_n^{2\beta}$,
\begin{eqnarray*}
\big\{ \sigma \in [\sigma_n, 2\sigma_n], \norm{\mu - \mu_\beta}_{\infty} \precsim \sigma_n^{\beta+1} \big\} \subset \bigg\{  \int f_0 \log \frac{f_0}{f_{\mu, \sigma}} \precsim \sigma_n^{2\beta}, \int f_0 \log \bigg(\frac{f_0}{f_{\mu, \sigma}}\bigg)^2 \precsim \sigma_n^{2\beta}\bigg\}.
\end{eqnarray*}

Since $\mu_\beta \in C^{\beta+1}[0,1]$, from Section 5.1 of \cite{van2009adaptive},
\begin{align*}
&\Pi_{\mu}(\norm{\mu - \mu_\beta}_{\infty} \leq 2\delta_n) \geq C_4\exp\bigg\{-C_5(1/\delta_n)^{\frac{1}{\beta+1}}\log\bigg(\frac{1}{\delta_n}\bigg)^{2\vee q} \bigg\}(C_6/\delta_n)^{(p+1)/(\beta+1)},
\end{align*}
for $\delta_n \to 0$ and constants $C_4, C_5, C_6 > 0$.  Letting $\delta_n = \sigma_n^{\beta+1}$, we obtain
\begin{align*}
&\Pi_{\mu}(\norm{\mu - \mu_\beta}_{\infty} \leq 2\delta_n) \geq \exp\bigg\{-C_7\bigg(\frac{1}{\sigma_n}\bigg)\log\bigg(\frac{1}{\sigma_n^{\beta+1}}\bigg)^{2\vee q}\bigg\},
\end{align*}
for some constant $C_7 > 0$.
Since $\sigma \sim IG(a_{\sigma}, b_{\sigma})$,
we have
\begin{align*}
\Pi_{\sigma}(\sigma \in [\sigma_n, 2\sigma_n ]) &= \frac{b_{\sigma}^{a_{\sigma}}}{\Gamma(a_{\sigma})}\int_{\sigma_n}^{2\sigma_n}x^{-(a_{\sigma}+1)} e^{-b_{\sigma}/x}dx \\
&\geq \frac{b_{\sigma}^{a_{\sigma}}}{\Gamma(a_{\sigma})}\int_{\sigma_n}^{2\sigma_n} e^{-2b_{\sigma}/x}dx \\
&\geq \frac{b_{\sigma}^{a_{\sigma}}}{\Gamma(a_{\sigma})}\sigma_n\exp\{-b_{\sigma}/\sigma_n \}\\
&\geq \exp\{-C_8/\sigma_n \},
\end{align*}
for some constant $C_8> 0$. Hence
\begin{align*}
\Pi \{\sigma \in [\sigma_n, 2\sigma_n ], \norm{\mu - \mu_\beta}_{\infty} \precsim \sigma_n^{\beta+1}\} &\geq \exp\bigg\{-C_7\bigg(\frac{1}{\sigma_n}\bigg)\log\bigg(\frac{1}{\sigma_n^{\beta+1}}\bigg)^{2\vee q}\bigg\}\exp\{-C_8/\sigma_n \}\\
&\ge \exp\bigg\{-2C_9\bigg(\frac{1}{\sigma_n}\bigg)\log\bigg(\frac{1}{\sigma_n^{\beta+1}}\bigg)^{2\vee q}\bigg\}. 
\end{align*}
Then (\ref{eq3}) will be satisfied with $\widetilde{\epsilon}_n = n^{-\beta/(2\beta+1)}\log^{t_1}(n)$, where $t_1=\beta(2\vee q)/(2\beta+1)$ and some $C_9 > 0$.
Next we construct a sequence of subsets $\mathcal{F}_n$ such that \eqref{eq1} and \eqref{eq2} are satisfied with
$\bar{\epsilon}_n = n^{-\beta/(2\beta+1)}\log ^{t_2}n$ and $\widetilde{\epsilon}_n$ for some global constant $t_2 > 0$.

Now we construct the sieves for $\mathcal{F}$. Letting $\mathbb{H}_1^a$ denote the unit ball of RKHS of the Gaussian process with rescaled parameter $a$ and  $\mathbb{B}_1$ denote the unit ball of
$C[0,1]$ and given positive sequences $M_n, r_n$, define
\begin{eqnarray*}
B_n =  \cup_{a < r_n}(M_n \mathbb{H}^a_1) + \bar{\delta}_n\mathbb{B}_1,
\end{eqnarray*}
as in \cite{van2009adaptive}, with $\bar{\delta}_n =\bar{\epsilon}_nl_n/K_1, K_1 = 2(2/\pi)^{1/2}$ and let
\begin{eqnarray*}
\mathcal{F}_n = \{f_{\mu, \sigma}: \mu \in B_n, l_n < \sigma < h_n \}.
\end{eqnarray*}
First we need to calculate $N(\bar{\epsilon}_n, \mathcal{F}_n, \norm{\cdot}_1)$. Observe that for $\sigma_2 >\sigma_1 > \sigma_2/2$,
\begin{eqnarray*}
\norm{f_{\mu_1, \sigma_1} - f_{\mu_2, \sigma_2}}_1 \leq \bigg(\frac{2}{\pi}\bigg)^{1/2}\frac{\norm{\mu_1 - \mu_2}_{\infty}}{\sigma_1} + \frac{3(\sigma_2 - \sigma_1)}{\sigma_1}.
\end{eqnarray*}

Taking $\kappa_n =\min\{\bar{\epsilon}_n/6, 1\}$ and $\sigma_m^n = l_n (1+ \kappa_n)^m, m \geq 0$, we obtain a partition of $[l_n, h_n]$ as $l_n=\sigma_0^n < \sigma_1^n < \cdots < \sigma_{m_{n}-1}^n < h_n \leq \sigma_{m_n}^n$ with
\begin{eqnarray}\label{eq:entropy1}
m_n= \bigg(\log \frac{h_n}{l_n}\bigg) \frac{1}{\log( 1+\kappa_n)} +1.
\end{eqnarray}
One can show that $3(\sigma_m^n -\sigma_{m-1}^n)/\sigma_{m-1}^n = 3\kappa_n \leq \bar{\epsilon}_n/2$. Let
$\{\widetilde{\mu}_k^n, k=1, \ldots, N(\bar{\delta}_n, B_n, \norm{\cdot}_{\infty})\}$  be
a $\bar{\delta}_n$-net of $B_n$. Now consider the set
\begin{eqnarray} \label{eq:net}
\{(\widetilde{\mu}_k^n, \sigma_{m}^n): k=1, \ldots, N(\bar{\delta}_n, B_n, \norm{\cdot}_{\infty}),\, 0\leq m\leq m_n \}.
\end{eqnarray}
Then for any $f= f_{\mu, \sigma} \in \mathcal{F}_n$, we can find $(\widetilde{\mu}_k^n, \sigma_{m}^n)$ such that
$\norm{\mu - \widetilde{\mu}_k^n}_{\infty} < \bar{\delta}_n$. In addition, if one has $\sigma \in (\sigma_{m-1}^{n}, \sigma_m^{n}]$,
then
\begin{eqnarray*}
\norm{f_{\mu, \sigma} - f_{\mu_k^n, \sigma^n_m}}_1 \leq \bar{\epsilon}_n.
\end{eqnarray*}
Hence the set in (\ref{eq:net}) is an $\bar{\epsilon}_n$-net of $\mathcal{F}_n$ and its covering number is given by
\begin{eqnarray*}
m_nN(\bar{\delta}_n, B_n, \norm{\cdot}_{\infty}).
\end{eqnarray*}
From the proof of Theorem 3.1 in \cite{van2009adaptive}, for any $M_n, r_n$ with $r_n > 0$, we obtain
\begin{eqnarray}\label{eq:entropy2}
\log N(2\bar{\delta}_n, B_n, \norm{\cdot}_{\infty}) \leq K_2 r_n \bigg( \log \bigg(\frac{M_n}{\bar{\delta}_n}\bigg)\bigg)^{2}.
\end{eqnarray}
Again from the proof of Theorem 3.1 in \cite{van2009adaptive}, for $r_n > 1$ and for $M_n^2 > 16K_3r_n (\log (r_n / \bar{\delta}_n))^2$, we have
\begin{eqnarray}\label{eq:compact1}
\bbP(W^A \notin B_n) \leq \frac{K_4r_n^p e^{-K_5r_n\log^q r_n}}{K_5\log^q r_n} + \exp\{-M_n^2/8\},
\end{eqnarray}
for constants $K_3, K_4, K_5 > 0$.

Next we calculate $\bbP(\sigma \notin [l_n, h_n])$. Observe that
\begin{align}
\bbP(\sigma \notin [l_n, h_n ]) &= \bbP(\sigma^{-1} < h_n^{-1}) + \bbP(\sigma^{-1} > l_n^{-1})\nonumber\\
 &\leq \sum_{k=\alpha_{\sigma}}^{\infty}\frac{e^{-b_{\sigma}h_n^{-1}}(b_{\sigma}h_n^{-1})^k}{k!} + \frac{b_{\sigma}^{a_{\sigma}}}{\Gamma(a_{\sigma})}\int_{l_n^{-1}}^{\infty} e^{-b_{\sigma}x/2}dx \nonumber \\
&\leq e^{-a_{\sigma}\log(h_n)} + \frac{b_{\sigma}^{a_{\sigma}}}{\Gamma(a_{\sigma})}e^{-b_{\sigma}l_n^{-1}/2}.\label{eq:compact2}
\end{align}
Thus with $h_n = O(\exp\{n^{1/{(2\beta+1)}}(\log n)^{2t_1}\})$, $ l_n = O(n^{-1/{(2\beta+1)}}(\log n)^{-2t_1})$, $ r_n = O(n^{1/{(2\beta+1)}}(\log n)^{2t_1})$, $ M_n = O(n^{1/{(2\beta+1)}}(\log n)^{t_1+1})$, (\ref{eq:compact1}) and (\ref{eq:compact2}) implies
\begin{eqnarray*}
\Pi(\mathcal{F}_n^c)= \exp\{-K_6n\widetilde{\epsilon}^2_n\},
\end{eqnarray*}
for some constant $K_6 > 0$, which guarantees that (\ref{eq2}) is satisfied with $\widetilde{\epsilon}_n = n^{-\beta/{(2\beta+1)}}(\log n)^{t_1}$.

Also with $\bar{\epsilon}_n = n^{-\beta/{(2\beta+1)}}(\log n)^{t_1+1}$,  it follows from
(\ref{eq:entropy1}) and (\ref{eq:entropy2}) that
\begin{eqnarray*}
\log N(\bar{\epsilon}_n, \mathcal{F}_n, \norm{\cdot}_1) \leq K_7 n^{1/{(2\beta+1)}}(\log n)^{2t_1+2},
\end{eqnarray*}
for some constant $K_7 > 0$. Hence $\max\{\bar{\epsilon}_n, \widetilde{\epsilon}_n\} = n^{-\beta/{(2\beta+1)}}(\log n)^{t_1+1}$.
\end{proof}

\subsection{Proof of Theorem \ref{thm:vi_bound}}\label{sec:vi_bound}
In this section, we present the detailed proof of the high probability bound for KL divergence between the true posterior and its  $\alpha$-VB approximation in the case of the GP-IVI.\\
\\
\textbf{Theorem 4.1. }Under assumptions \textbf{B1} through \textbf{B5} it hold that $m^*_n(\m Q_n)=\min_{q\in \m Q_n} \left\{D[q||p(\cdot \mid Y^{(n)})]\right\}$ is bounded in probability with respect to the data generating distribution. Formally, given any $\varepsilon>0$, there exists $M_\varepsilon, N_\varepsilon >0$ such that for $n\geq N_\varepsilon$, we have $\bbP_{\theta^*}^{(n)}(m^*_n(\m Q_n)> M_\varepsilon) \leq \varepsilon$.\\
\\
The objective $m_n^*(\m Q_n)$ can be bounded above by $D[q||p(Y^{(n)} \mid \theta)]$ for any $q\in \m Q_n$. Choosing $q$ as a particular univariate Gaussian centered at the true parameter with variance satisfying our assumptions \textbf{B1}-\textbf{B5} allows us to bound the KL divergence between the true posterior $p(Y^{(n)} \mid \theta)$ in high $\bbP_{\theta^*}^{(n)}$-probability. 

\begin{proof}
It follows from the definition of $m_n^*(\m Q_n)$ that for any $q\in \m Q_n$ 
\begin{align*}
    m_n^*(\m Q_n)\leq D(q||p(\cdot \mid Y^{(n)})).
\end{align*}
Choose $\mu_n$ to be the quantile function of the distribution $N(\theta^*,\sigma_n^2)$. Define the variational distribution
\begin{align*}
    q_n(\theta)= \int \phi_\sigma(\theta-\mu_n(u))du,
\end{align*} where $\sigma_n$ satisfies assumption \textbf{B2}. 
By change of measure,
\begin{align*}
  \int \phi_\sigma(\theta-\mu_n(u))du=\int \phi_\sigma(\theta-t)\phi_{\sigma_n}(t-\theta^*)dt = N(\theta;\theta^*,\sigma^2+\sigma_n^2).
\end{align*} 
Therefore  $q_n(\theta)=N(\theta;\theta^*,\sigma^2+\sigma_n^2) \in \m Q_n$. Denote by $\bbE_n$ the mean respect to $q_n$. Expanding $D(q_n||p(Y^{(n)} \mid \theta))$,
\begin{align*}
    \bbE_n\left[ \log \frac{q_n(\theta)}{p(Y^{(n)} \mid \theta)(\theta)} \right]= \bbE_n[\log q_n] + \bbE_n[U(\theta)] + \log m(Y^{(n)}) - \bbE_n\left[L_n(\theta,\theta^*)\right],
\end{align*}
where $L_n(\theta,\theta^*)= \sum_{i=1}^n\ell_i(\theta,\theta^*)$. Since the sum of $O_p(1)$ terms is $O_p(1)$, it suffices to show that each of the terms in the above sum is $O_p(1)$. The first term $\bbE_n[\log q_n]$, the differential entropy of $q_n$, is a constant and is $O_p(1)$. A straight forward application of Markov's inequality along with the fact that $\bbE_{\theta^*}^{(n)}[ m(Y^{(n)})]=1$ shows that $\log m(Y^{(n)})$ is $O_p(1)$.

Next, expand each of the functions  $D(\theta^*||\theta)$, $\mu_2(\theta^*||\theta)$, and $U(\theta)$ using a multivariate Taylor expansion around $\theta^*$. Applying assumptions \textbf{B4} and \textbf{B5} shows 
 \begin{align}
    \bbE_n[U(\theta)]\leq C_1 (\sigma^2 + \sigma_n^2),\nonumber\\
    \bbE_n[\mu_2(\theta^*||\theta)]\leq C_2 (\sigma^2 + \sigma_n^2),\label{mu2.ub}\\
    \bbE_n[D(\theta^*||\theta)] \leq C_u (\sigma^2 + \sigma_n^2), \label{kl.ub}\\
    \bbE_n[D(\theta^*||\theta)] \geq C_\ell (\sigma^2 + \sigma_n^2). \label{kl.lb}
\end{align}
Markov's inequality shows that $U(\theta)$ is $O_p(1)$. We will use Chebychev's inequality to show $\bbE_n\left [\sum_{i=1}^n\ell_i(\theta,\theta^*)\right]$ is $O_p(1)$. Given $\varepsilon >0$, choose $\delta= \left [C_2c_0\slash (\varepsilon C_\ell)^2 \right]^{1\slash 2}$. Using (\ref{mu2.ub})-(\ref{kl.lb}) and noting that $-\bbE^{(n)}_{\theta^*}\{L_n(\theta,\theta^*)\} =nD(\theta^*||\theta)$, we have
\small
\begin{align*}
\mb P_{\theta^*}^{(n)}\left\{\bbE_n[ L_n(\theta,\theta^*)]  \leq -C_u(1+\delta)n(\sigma^2+ \sigma_n^2) \right\}
& \leq 
\mb  P_{\theta^*}^{(n)}\left\{ \bbE_n[ L_n(\theta,\theta^*)] \leq -(1+\delta)n\bbE_n[D(\theta^*||\theta)] \right\} \\
&\leq \mb  P_{\theta^*}^{(n)} \left\{  \frac{1}{\sqrt{n}}\bbE_n[ L_n(\theta,\theta^*)- \bbE^{(n)}_{\theta^*}\{L_n(\theta^*,\theta)\} ]  \leq -\delta \sqrt{n}\bbE_n[D(\theta^*||\theta)] \right\}\\
&\leq \frac{\text{Var}_{\theta^*}^{(n)}\left( \bbE_n[ \ell_1(\theta,\theta^*)] \right)}{\delta^2 n \left(\bbE_n[ D(\theta^*||\theta)] \right)^2}
\leq \frac{\bbE_n[\mu_2(\theta^*||\theta^*)]}{\delta^2 n \left(\bbE_n[D(\theta^*||\theta)]\right)^2}\\
&\leq \frac{C_2(\sigma^2+\sigma_n^2)}{\delta^2nC_{\ell}(\sigma^2+\sigma_n^2)^2}\leq \frac{C_2}{\delta^2nC_\ell^2(\sigma^2+\sigma_n^2)}\leq\frac{C_2}{\delta^2n C_\ell^2 \sigma_n^2}. 
\end{align*}
\normalsize
Applying assumption \textbf{B2} we have $c_0^{-1\slash 2}n^{-1\slash2} \leq \sigma_n \leq n^{-1\slash 2}$. This gives
\small
\begin{align*}
\bbP_{\theta^*}^{(n)}\left\{ \int L_n(\theta,\theta^*)q_n(\theta)d\theta \leq -2C_u(1+(C_2c_0\slash(\varepsilon C_\ell^2))^{1\slash 2})\right\}\leq 
 \bbP_{\theta^*}^{(n)}\left\{ \int L_n(\theta,\theta^*)q_n(\theta)d\theta \leq -C_u(1+\delta)n(\sigma^2+\sigma_n^2)\right\}  \leq \varepsilon.
\end{align*}
\normalsize
Thus $\bbE_n[L_n(\theta,\theta^*)]$ is $O_p(1)$. This completes the proof.
\end{proof}

\subsection{Proof of Theorem \ref{thm:riskbound}}\label{sec:riskbound}

In this section, we present the detailed proof of the Bayesian risk bound for $\alpha$-variational inference in the case of the GP-IVI model. We also present a proof of the corollary for the Hellinger risk bound. The main theorem and the lemmas are restated here for convenience. Our risk bound is based of the following theorem, 
\begin{theorem}[\cite{yang2020alpha}]
For any $\zeta\in (0,1)$, it holds with $\bbP_{\theta^*}^{(n)}$-probability at least $(1-\zeta)$ that for any probability measure $q\in \m Q$ with $q \ll p_\theta$,
\begin{align*}
\int \frac{1}{n}D_\alpha[p_\theta^{(n)}||p_{\theta^*}^{(n)}]\widehat{q}(\theta)d\theta \leq \frac{\alpha\Psi(q) + \log(1\slash \zeta)}{n(1-\alpha)}. 
\end{align*}
\end{theorem}

The GP-IVI risk bound is stated as follows.

\textbf{Theorem 4.2. }Assume $\widehat{q}_{\mu,\sigma}$ satisfies \eqref{avb.obj} and $\widehat{q}_{\mu,\sigma} \ll p_\theta$. It holds with $ \bbP_{\theta^*}^{(n)}$-probability at least $1-2\slash[(D-1)^2(1+O(n^{-2}))n\varepsilon^2]$ that,
\begin{align*}
 \int \frac{1}{n}D_{\alpha}^{(n)}(\theta,\theta^*)\widehat{q}_{\mu,\sigma}(\theta)d\theta  \leq \frac{D\alpha  }{1-\alpha}\varepsilon^2 + \frac{1}{n(1-\alpha)}\log\left\{ \bbP_\theta\left[ B_n(\theta^*,\varepsilon)\right]^{-1}\right\} + O(n^{-1}). 
\end{align*}

The desired risk bound follows from bounding the right hand side of Theorem 3.2 of \cite{yang2020alpha} 
\begin{align*}
    \frac{\alpha}{n(1-\alpha)} \Psi(q_{\mu,\sigma})&:= \frac{\alpha}{n(1-\alpha)}\left [ \int q_{\mu,\sigma}(\theta) \log \frac{p(Y^{(n)}\mid \theta^*)}{ p(Y^{(n)}\mid \theta)} d\theta +\frac{1}{\alpha}D(q_{\mu,\sigma}||p_\theta) \right]
\end{align*}
in high $\bbP_{\theta^*}^{(n)}$-probability in terms of the local Bayesian complexity $\log \bbP_\theta(B_n(\theta^*,\varepsilon))$. By choosing a particular member of the variational family we can bound both the likelihood ratio integral as well as the KL divergence between the prior and the variational approximation. The relation between the variational distribution and the local Bayesian complexity come from the KL divergence term.

\begin{proof}
 We will construct a special choice of $\mu$ as follows. Denote $p_\theta(\theta)=f_0(\theta)$. Let $B_n(\theta^*,\varepsilon)$ be as in (\ref{kl.ball}). Define the truncated densities
 \begin{align*}
    \widetilde{f}_0(t)=\frac{f_0(t)I_{B_n(\theta^*,\varepsilon)}(t)}{\int_{B_n(\theta^*,\varepsilon)} f_0(u)du}=\frac{f_0(t)I_{B_n(\theta^*,\varepsilon)}(t)}{\bbP_\theta(B_n(\theta^*,\varepsilon))},\quad 
    \widetilde{f}_\beta(t)=\frac{f_\beta(t)I_{B_n(\theta^*,\varepsilon)}(t)}{\int_{B_n(\theta^*,\varepsilon)} f_\beta(u)du}, 
\end{align*}
where $f_\beta$ is constructed by procedure \eqref{eq:construction} such that $\Vert \phi_\sigma*f_\beta-f_0\Vert_\infty=O(\sigma^\beta)$ along with its associated distribution functions
\begin{align*}
    \widetilde{F}_0(t)=\int_{(-\infty,t]\cap B_n(\theta^*,\varepsilon)}\widetilde{f}_0(t)dt,\quad
    \widetilde{F}_\beta(t)=\int_{(-\infty,t]\cap B_n(\theta^*,\varepsilon)}\widetilde{f}_\beta(t)dt.
\end{align*}
Define the quantile function of $\widetilde{F}_\beta$ as $\widetilde{\mu}(t)=\widetilde{F}^{-1}_\beta(t)$. This can be used to define the variational density
\begin{align*}
    q_{\widetilde{f}_\beta,\sigma}(\theta) = \int_{[0,1]} \phi_\sigma(\theta-\widetilde{\mu}(\eta))d\eta =\int_{-\infty}^\infty \phi_\sigma(\theta-t)\widetilde{f}_\beta(t)dt 
    = \phi_\sigma * \widetilde{f}_\beta(\theta),
\end{align*}
with $\sigma>0$ a bandwidth that will be specified later in the proof. The main tool for the proof will be from Proposition \ref{approximation} 
\begin{align}
    q_{\widetilde{f}_\beta,\sigma}(\theta) = \phi_\sigma * \widetilde{f}_\beta(\theta) \leq \widetilde{f}_0(\theta)(1+D(\theta)O(\sigma^\beta)). \label{q.bound}
\end{align} 
Denote $M_D=\sup_{B_n(\theta^*,\varepsilon)}D(\theta)$ and $K_\beta(\sigma)=1+M_D O(\sigma^\beta)$. We will now bound the model-fit term. Denote the random variable
\begin{align*}
    H(Y^{(n)},\widetilde{f}_\beta,\sigma)&=\int q_{\widetilde{f}_\beta,\sigma}(\theta) \log [p(Y^{(n)}\mid \theta^*)\slash p(Y^{(n)}\mid \theta)] d\theta.
\end{align*}
The mean and variance (with respect to the data generating distribution) of the model-fit term are bounded by applying \eqref{q.bound},
\begin{align*}
   \bbE_{\theta^*}^{(n)}[  H(Y^{(n)},\widetilde{f}_\beta,\sigma)]  &=   \int  D[p(Y^{(n)}\mid \theta^*)|| p(Y^{(n)}\mid \theta)] q_{\widetilde{f}_\beta,\sigma}(\theta)d\theta \\
   &\leq \int  D[p(Y^{(n)}\mid \theta^*)|| p(Y^{(n)}\mid \theta)] \widetilde{f}_0(\theta)(1+D(\theta)O(\sigma^\beta)) d\theta\\
   &\leq K_\beta(\sigma) \int_{B(\theta^*,\varepsilon)} D[p(Y^{(n)}\mid \theta^*)|| p(Y^{(n)}\mid \theta)] \frac{f_0(\theta)}{\bbP_\theta[B_n(\theta^*,\varepsilon)]} d\theta\\
   &\leq K_\beta(\sigma)n\varepsilon^2,
   \end{align*}
and
\begin{align*}
   \text{Var}_{\theta^*}^{(n)}[H(Y^{(n)},\widetilde{\mu},\sigma)] &\leq \int  V[p(Y^{(n)}\mid \theta^*)|| p(Y^{(n)}\mid \theta)] q_{\widetilde{f}_\beta,\sigma}(\theta)d\theta \\
   &\leq \int  V[p(Y^{(n)}\mid \theta^*)|| p(Y^{(n)}\mid \theta)] \widetilde{f}_0(\theta)(1+D(\theta)O(\sigma^\beta)) d\theta\\
   &\leq K_\beta(\sigma) \int_{B(\theta^*,\varepsilon)} V[p(Y^{(n)}\mid \theta^*)|| p(Y^{(n)}\mid \theta)] \frac{f_0(\theta)}{\bbP_\theta[B_n(\theta^*,\varepsilon)]} d\theta\\
   &\leq K_\beta(\sigma)n\varepsilon^2.
\end{align*}
It follows from Chebyshev's inequality that with $\bbP_{\theta^*}^{(n)}$-probability at least $1-1\slash[(D-1)^2K_\beta(\sigma)n\varepsilon^2]$
\begin{align*}
    \int q_{\widetilde{f}_\beta,\sigma}(\theta) \log\left[ \frac{p(Y^{(n)}\mid \theta^*)}{ p(Y^{(n)}\mid \theta)}\right] d\theta \leq  DK_\beta(\sigma)n\varepsilon^2. 
\end{align*}
Next we will bound the regularization in terms of the local Bayesian complexity. 
Using \eqref{q.bound} we can bound the KL divergence,
\begin{align*}
    D[q_{\widetilde{f}_\beta,\sigma}||p_\theta] &=  \int q_{\widetilde{f}_\beta,\sigma}(\theta) \log\left[ \frac{q_{\widetilde{f}_\beta,\sigma}(\theta)}{f_0(\theta)}\right] d\theta 
\leq  \int \log\left[  \frac{\widetilde{f}_0(\theta)(1+O(D(\theta)\sigma^\beta))}{f_0(\theta)}\right ] \widetilde{f}_0(\theta)(1+O(D(\theta)\sigma^\beta)) d\theta.  \\
  \end{align*}
  Expanding $\widetilde{f}_0(\theta)$ and making use of the convention $I_{B_n(\theta^*,\varepsilon)}(\theta)\log(I_{B_n(\theta^*,\varepsilon)}(\theta))=0$ for $\theta\notin B_n(\theta^*,\varepsilon)$ we have 
  \small
 \begin{align*}   
&\int \log\left[  \frac{f_0(\theta)I_{B_n(\theta^*,\varepsilon)}(1+O(D(\theta)\sigma^\beta))}{f_0(\theta)\bbP_\theta[B_n(\theta^*,\varepsilon)] }\right ] \frac{f_0(\theta)I_{B_n(\theta^*,\varepsilon)}}{\bbP_\theta[B_n(\theta^*,\varepsilon)]}(1+O(D(\theta)\sigma^\beta)) d\theta  \\
& =\int_{B_n(\theta^*,\varepsilon)} \log\left[  \frac{(1+O(D(\theta)\sigma^\beta))}{\bbP_\theta[B_n(\theta^*,\varepsilon)] }\right ] \frac{f_0(\theta)}{\bbP_\theta[B_n(\theta^*,\varepsilon)]}(1+O(D(\theta)\sigma^\beta)) d\theta  \\ 
    &\leq  K_\beta(\sigma)\log \left [ \frac{K_\beta(\sigma)}{\bbP_\theta(B_n(\theta^*,\varepsilon))} \right] \int_{B_n(\theta^*,\varepsilon)}\frac{f_0(\theta)}{\bbP_\theta[B_n(\theta^*,\varepsilon)]} d\theta \\
    &=  K_\beta(\sigma)\log \left [ \frac{K_\beta(\sigma)}{\bbP_\theta(B_n(\theta^*,\varepsilon))} \right].
\end{align*}
\normalsize
 Combining the bounds from both parts, we have with probability at least $1-1\slash [(D-1)^2K_\beta(\sigma)n\varepsilon^2]$ that 
\begin{align*}
    \Psi(q_{\widetilde{f}_\beta,\sigma})\leq  DK_\beta(\sigma)n\varepsilon^2 + \alpha^{-1} K_\beta(\sigma)\log K_\beta(\sigma)  + \alpha^{-1} K_\beta(\sigma)\log\left\{ \bbP_\theta[B_n(\theta^*,\varepsilon)]^{-1}\right \}. 
\end{align*}
Choosing $\zeta=1\slash [(D-1)^2K_\beta(\sigma)n\varepsilon^2]$. It follows from the union bound for probabilities, we have with probability at least $1-2\slash [(D-1)^2K_\beta(\sigma)n\varepsilon^2]$ that 
\small
\begin{align*}
 &\int \frac{1}{n}D_{\alpha}^{(n)}(\theta,\theta^*)\widehat{q}_{\mu,\sigma}(\theta)d\theta \leq \\
 &\frac{\alpha  DK_\beta(\sigma)n\varepsilon^2 + K_\beta(\sigma)\log K_\beta(\sigma) + K_\beta(\sigma)\log\left\{ \bbP_\theta[B_n(\theta^*,\varepsilon)]^{-1}\right \}  +\log((D-1)^2K_\beta(\sigma)n\varepsilon^2)}{n(1-\alpha)}\\
 & \leq K_\beta(\sigma) \left(\frac{D\alpha  }{1-\alpha}\varepsilon^2 + \frac{1}{n(1-\alpha)}\log\left\{ \bbP_\theta[B_n(\theta^*,\varepsilon)]^{-1}\right \} +O(n^{-1}) \right).
\end{align*}
\normalsize
Recall that $K_\beta(\sigma)=1+O(\sigma^\beta)$. Choosing $\sigma=n^{-2\slash \beta}$ gives 
\begin{align*}
 \int \frac{1}{n}D_{\alpha}^{(n)}(\theta,\theta^*)\widehat{q}_{\mu,\sigma}(\theta)d\theta &\leq K_\beta(\sigma) \left(\frac{D\alpha  }{1-\alpha}\varepsilon^2 + \frac{1}{n(1-\alpha)}\log\left\{ \bbP_\theta[B_n(\theta^*,\varepsilon)]^{-1}\right \} +O(n^{-1}) \right)\\
 &\leq \frac{D\alpha  }{1-\alpha}\varepsilon^2 + \frac{1}{n(1-\alpha)}\log\left\{ \bbP_\theta[B_n(\theta^*,\varepsilon)]^{-1}\right \} + O(n^{-1}) + O(n^{-2}).
\end{align*}
\end{proof}

\textbf{Corollary 4.1. }Suppose the prior density $p_\theta$ satisfies Assumption {\bf A1} and  $\widehat{q}$ satisfies \eqref{avb.obj}. It holds with probability tending to one as $n\to \infty$ that,
\begin{align*}
    \left\{\int h^2(p(\cdot\mid \theta), p(\cdot\mid \theta^*)) \widehat{q}_{\mu,\sigma}(\theta)d\theta\right\}^{1\slash 2} \leq  O(n^{-1}),
\end{align*}
demonstrating that the risk bound is parametric even when a flexible class of variational approximation is used. 
\begin{proof}
For IID data $n^{-1}D_{\alpha}^{(n)}(\theta,\theta^*)=D_\alpha[p_\theta||p_{\theta^*}]$. Applying Theorem 4.2 with $\varepsilon=n^{-1}$ and Assumption {\bf A1} yields,
\begin{align*}
   \int \frac{1}{n}D_{\alpha}^{(n)}(\theta,\theta^*)\widehat{q}_{\mu,\sigma}(\theta)d\theta &\leq \frac{D\alpha }{1-\alpha}\varepsilon^2 + \frac{1}{n(1-\alpha)}\log\left\{ \bbP_\theta[B_n(\theta^*,\varepsilon)]^{-1}\right \} + O(n^{-1})\\
   &\leq \frac{D\alpha-1 }{n^2(1-\alpha)} + O(n^{-1}) = O(n^{-2}) + O(n^{-1}). 
\end{align*}
Combining the above with the fact that $\max\{1,(1-\alpha)^{-1}\alpha\}h^2(p,q)\leq D_\alpha[p||q]$ competes the proof.
\end{proof}

\section{Auxiliary results}\label{sec:auxiliary}

In this section, we summarize results used in the proofs of main theorems in the main document. First to guarantee that the model \eqref{eq:nl-lvm} leads to the optimal rate of convergence, we start from deriving sharp bounds for the Hellinger distance between $f_{\mu_1, \sigma_1}$ and $f_{\mu_2, \sigma_2}$ for $\mu_1, \mu_2 \in C[0, 1]$ and $\sigma_1, \sigma_2 > 0$. We summarize the result in the following Lemma \ref{lem:hellinger}.
\begin{lem}\label{lem:hellinger}
For $\mu_1, \mu_2 \in C[0, 1]$ and $\sigma_1, \sigma_2 > 0$,
\begin{eqnarray}
h^2(f_{\mu_1,\sigma_1}, f_{\mu_2, \sigma_2}) \leq 1- \sqrt{\frac{2\sigma_1\sigma_2}{\sigma_1^2 + \sigma_2^2}}\exp\bigg\{-\frac{\norm{\mu_1 - \mu_2}_{\infty}^2}{4(\sigma_1^2 + \sigma_2^2)}\bigg\}.
\end{eqnarray}
\end{lem}
\begin{proof}
Note that by H\"{o}lder's inequality,
\begin{align*}
f_{\mu_1, \sigma_1}(y)f_{\mu_2, \sigma_2}(y) \geq \bigg\{\int_{0}^{1} \sqrt{\phi_{\sigma_1}(y - \mu_1(x))} \sqrt{\phi_{\sigma_2}(y - \mu_2(x))}dx\bigg\}^2.
\end{align*}
Hence,
\begin{align*}
h^2(f_{\mu_1,\sigma_1}, f_{\mu_2, \sigma_2}) &\leq \int\bigg[\int_{0}^{1}\phi_{\sigma_1}(y-\mu_1(x))dx  + \int_{0}^{1}\phi_{\sigma_2}(y-\mu_2(x))dx \\
&~~~-2\int_{0}^{1}\sqrt{\phi_{\sigma_1}(y - \mu_1(x))} \sqrt{\phi_{\sigma_2}(y - \mu_2(x))}dx\bigg]dy.
\end{align*}
By changing the order of integration (applying Fubini's theorem since the function within the integral is jointly integrable) we get
\begin{align*}
h^2(f_{\mu_1,\sigma_1}, f_{\mu_2, \sigma_2}) &\leq \int_{0}^{1}h^2(f_{\mu_1(x),\sigma_1}, f_{\mu_2(x), \sigma_2})dx \\
&= \int_{0}^{1}\bigg[1- \sqrt{\frac{2\sigma_1\sigma_2}{\sigma_1^2 + \sigma_2^2}}\exp\bigg\{-\frac{(\mu_1(x) - \mu_2(x))^2}{4(\sigma_1^2 + \sigma_2^2)}\bigg\}\bigg]dx\\
&\leq 1- \sqrt{\frac{2\sigma_1\sigma_2}{\sigma_1^2 + \sigma_2^2}}\exp\bigg\{-\frac{\norm{\mu_1 - \mu_2}_{\infty}^2}{4(\sigma_1^2 + \sigma_2^2)}\bigg\}.
\end{align*}

\end{proof}

\begin{remark}
When $\sigma_1 = \sigma_2  = \sigma$, $h^2(f_{\mu_1,\sigma}, f_{\mu_2, \sigma}) \leq 1 - \exp\big\{\norm{\mu_1 -\mu_2}_{\infty}^2/ 8 \sigma^2\big\}$, which implies that $h^2(f_{\mu_1,\sigma}, f_{\mu_2, \sigma}) \precsim \norm{\mu_1 -\mu_2}_{\infty}^2/\sigma^2$.
\end{remark}

\begin{remark}
The standard inequality $h^2(f_{\mu_1,\sigma_1}, f_{\mu_2, \sigma_2}) \leq \norm{f_{\mu_1,\sigma_1}- f_{\mu_2, \sigma_2}}_1$ relating the Hellinger distance to the total variation distance leads to the cruder bound
\begin{eqnarray*}
h^2(f_{\mu_1,\sigma_1}, f_{\mu_2, \sigma_2}) \leq C_1 \frac{\norm{\mu_1 -\mu_2}_{\infty}}{(\sigma_1 \wedge \sigma_2)} + C_2\frac{|\sigma_2 - \sigma_1|}{(\sigma_1 \wedge \sigma_2)},
 \end{eqnarray*}
which is linear in $\norm{\mu_1 -\mu_2}_{\infty}$. This bound is less sharp than what is obtained in Lemma \ref{lem:hellinger} and does not suffice for obtaining the optimal rate of convergence.
\end{remark}

In order to apply Lemma 8 in \cite{ghosal2007posterior} to control the Kullback--Leibler divergence between the true density $f_0$ and the model $f_{\mu, \sigma}$, we derive an upper bound for $\log \norm{f_0/f_{\mu, \sigma}}_{\infty}$ in Lemma \ref{lem:logsup}.
\begin{lem}\label{lem:logsup}
If $f_0$ satisfies Assumption \textbf{F2},
\begin{eqnarray}
\log \bigg\|\frac{f_0}{f_{\mu, \sigma}}\bigg\|_{\infty} \leq C + \frac{\norm{\mu - \mu_0}_{\infty}^2}{\sigma^2}
\end{eqnarray}
for some constant $C > 0$.
\end{lem}
\begin{proof}
Note that
\begin{align*}
f_{\mu, \sigma}(y) &= \frac{1}{\sqrt{2\pi}\sigma}\int_{0}^{1}\exp\bigg\{-\frac{(y-\mu(x))^2}{2\sigma^2}\bigg\}dx\\
&\geq  \frac{1}{\sqrt{2\pi}\sigma}\int_{0}^{1}\exp\bigg\{-\frac{(y-\mu_0(x))^2}{\sigma^2}\bigg\}dx \exp\bigg\{-\frac{\norm{\mu-\mu_0}_{\infty}^2}{\sigma^2}\bigg\}\\
&\geq C \phi_{\sigma/\sqrt{2}} * f_0 (y) \exp\bigg\{-\frac{\norm{\mu-\mu_0}_{\infty}^2}{\sigma^2}\bigg\}\\
&\geq C f_0(y) \exp\bigg\{-\frac{\norm{\mu-\mu_0}_{\infty}^2}{\sigma^2}\bigg\},
\end{align*}
where the last inequality follows from Lemma 6 of \cite{ghosal2007posterior} since $f_0$ is compactly supported by Assumption \textbf{F2}. This provides the desired inequality. 
\end{proof}
\begin{lem}\label{lem:approx}
Let $j\ge 0$ be the integer such that $\beta \in (2j, 2j+2]$, and the sequence of $f_j$ is constructed by the procedure in \eqref{eq:construction}.
Then we have $f_\beta = \sum_{i=0}^{j}(-1)^i {j+1 \choose i+1} \phi_{\sigma}^{(i)}* f_0$, where $\phi_{\sigma}^{(i)}* f_0 = \phi_{\sigma}*\cdots *\phi_{\sigma}*f_0$, the $i$-fold convolution of $\phi_{\sigma}$ with $f_0$.
\end{lem}

\begin{proof}
Consider $f_j$ constructed by (\ref{eq:construction}). When $j=1, f_1 = 2f_{0} - \phi_{\sigma}*f_0$, so the form holds. 
By induction, suppose this form holds for $j > 1$, then 
\begin{align*}
f_{j+1} &= f_0 - (\phi_{\sigma}*f_j - f_{j}) \\
&= f_0 +  \sum_{i=0}^{j}(-1)^{i+1} {j+1 \choose i+1} \phi_{\sigma}^{(i+1)}* f_0 + \sum_{i=0}^{j}(-1)^i {j+1 \choose i+1} \phi_{\sigma}^{(i)}* f_0 \\
&= (j+2)f_0 +  \sum_{i=1}^{j+1}(-1)^{i} {j+1 \choose i+1} \phi_{\sigma}^{(i)}* f_0 +  \sum_{i=1}^{j}(-1)^{i} {j+1 \choose i} \phi_{\sigma}^{(i)}* f_0\\
&= (j+2)f_0  + \sum_{i=1}^{j} (-1)^{i} \bigg( {{j+1 \choose i+1} + {j+1 \choose i}} \bigg) \phi_{\sigma}^{(i)}*f_0 + (-1)^{j+1} \phi_{\sigma}^{(i +1)}*f_0 \\
&= (j+2)f_0 + \sum_{i=1}^{j}(-1)^{i} {j+2 \choose i+1} \phi_{\sigma}^{(i)}* f_0 + (-1)^{j+1} \phi_{\sigma}^{(i +1)}*f_0\\
&= \sum_{i=0}^{j+1}(-1)^{i} {j+2 \choose i+1} \phi_{\sigma}^{(i)}* f_0.
\end{align*}
It holds for $j +1$, which completes the proof.
\end{proof}

\begin{lem}\label{lem:tail}
Let $f_0$ satisfy Assumptions \textbf{F1} and \textbf{F2}. With $A_\sigma = \{x: f_0(x) \ge \sigma^{H}\}$, we have
\begin{eqnarray}\label{eq:tail} 
\int_{A^c_\sigma}f_0(x)dx = O(\sigma^{2\beta}), \  \   \int_{A^c_\sigma} \phi_{\sigma}*f_{j}(x) dx = O(\sigma^{2\beta}),
\end{eqnarray}
for all non-negative integer $j$, sufficiently small $\sigma$ and sufficiently large $H$.
\end{lem}
\begin{proof}
Under Assumption \textbf{F2} there exists $(a, b) \subset [0, 1]$ such that ${A^c_\sigma} \subset [0, a) \cup (b, 1]$ if we choose $\sigma$ sufficiently small, so that $f_0 (x) \le \sigma^{H}$ for $x \in {A^c_\sigma}$. Therefore, $\int_{A^c_\sigma}f_0(x) \le \sigma^{H} \le O(\sigma^{2\beta})$ if we choose $H \ge 2\beta$. Using Proposition  \ref{approximation}, 
\begin{eqnarray*}
\int_{A^c_\sigma} \phi_{\sigma}* f_j(x) dx =  \int_{A^c_\sigma} f_0(x) \{1 + O(D(x)\sigma^{\beta})\} \le O(\sigma^{H}).
\end{eqnarray*}
With bounded $D(x)$ and $H \ge 2\beta$ it is easy to bound the second integral in \eqref{eq:tail} by $O(\sigma^{2\beta})$.
\end{proof}

\begin{lem}\label{lem:density}
Suppose $f_0$ satisfies Assumptions \textbf{F1} and \textbf{F2}. For $\beta >2$ and the integer $j$ such that $\beta \in (2j,2j+2]$, $f_\beta$ is a density function. 
\end{lem}

\begin{proof}
To show $f_\beta$ is a density function, it suffices to show $f_\beta$ is non-negative, since a simple calculation shows that $\int f_\beta = 1$ for $j\ge0$. Following the proof of Lemma 2 in \cite{kruijer2010adaptive}, we treat $\log f_0$ as a function in $C^2 [0,1]$ and obtain the same form of $\phi_{\sigma}*f_0$ as in \eqref{eq:approximation}. For small enough $\sigma$ we can find  $\rho_1 \in (0,1)$ very close to $0$ such that
\begin{eqnarray*}
 \phi_{\sigma}* f_0(x) = f_0(x) (1 + O(D^{(2)}(x)\sigma^2)) < f_0(x)(1 + \rho_1),
\end{eqnarray*}
 where $D^{(2)}$ contains $|l_1(x)|$ and $|l_2(x)|$ to certain power, so $D^{(2)}$ is bounded. Then we have
 \begin{eqnarray*}
 f_1(x) = 2f_0(x) - K_{\sigma} f_0(x) > 2f_0(x) - f_0(x) (1 + \rho_1) = f_0(x) (1 - \rho_1).
 \end{eqnarray*}
 Then we treat $\log{f_0}$ as a function with $\beta = 4$, $j = 1$. Similarly, we can get
 \begin{eqnarray*}
 \phi_{\sigma}* f_1(x) = f_0(x) (1 + O(D^{(4)}(x)\sigma^4)),
 \end{eqnarray*}
 where $D^{(4)}$ contains $|l_1(x)|, \dots, |l_4(x)|$. We can find $0 < \rho_2 < \rho_1$ such that $\phi_{\sigma}* f_1(x) < f_0(x) (1 + \rho_2)$, then can get
 \begin{eqnarray*}
 f_2 (x) = f_0(x) - (\phi_{\sigma}* f_1(x) - f_1(x)) > f_0(x)(1 - \rho_1 - \rho_2) > f_0(x)(1 - 2\rho_1).
 \end{eqnarray*}
 Continuing this procedure, we can get $f_j(x) > f_0(x) (1 - j\rho_1)$ with sufficiently small $\sigma$ and $1 - j\rho_1 \in (0,1)$ and it is close to 1. Then we show $f_j$ is non-negative. \\
\end{proof}

\begin{lem}\label{lem:KL}
 Let $f_0$ satisfy Assumptions \textbf{F1} and \textbf{F2} and let $j$ be the integer such that $\beta \in (2j, 2j+2]$. Then we show that the density $f_\beta$ obtained by (\ref{eq:construction}) satisfies \begin{eqnarray}\label{eq:KL}
\int f_0(x) \log{\frac{f_0(x)}{\phi_{\sigma}* f_\beta (x)}} = O(\sigma^{2\beta}),
\end{eqnarray}
for sufficiently small $\sigma$ and all $x \in [0,1]$.
\end{lem}
\begin{proof}
Again consider the set $A_\sigma = \{x: f_0(x) \ge \sigma^{H}\}$ with arbitrarily large $H$. We separate the Kullback--Leibler divergence into
\begin{align}\label{integral}
\int_{[0, 1]} f_0 \log{\frac{f_0}{\phi_{\sigma}* f_\beta}} &= \int_{[0, 1]\cap A_{\sigma}} f_0 \log{\frac{f_0}{\phi_{\sigma}* f_\beta}} +  \int_{[0, 1]\cap A^c_{\sigma}} f_0 \log{\frac{f_0}{\phi_{\sigma}* f_\beta}} \nonumber \\
&\le \int_{A_\sigma} \frac{(f_0 - \phi_{\sigma}*f_\beta)^2}{\phi_{\sigma}* f_\beta} +  \int_{A^c_\sigma}(\phi_{\sigma}*f_\beta - f_0) +  \int_{A^c_\sigma} f_0 \log{\frac{f_0}{\phi_{\sigma}* f_\beta}}. 
\end{align}

Under Assumption \textbf{F2} and by Remark 3 in \cite{ghosal1999posterior}, for small enough $\sigma$ there exists a constant $C$ such that $\phi_{\sigma}*f_0 \ge Cf_0$ for all $x \in [0, 1]$. Especially,   $f_0$ satisfies $\phi_{\sigma}*f_0 \ge f_0/3$ for $ x\in A^c_{\sigma}$. Also in the proof of Lemma \ref{lem:density} we can find $\rho \in (0,1)$ such that $f_\beta > \rho f_0$. Then, on set $A_\sigma$ with sufficiently small $\sigma$, we have
\begin{eqnarray*}
\phi_{\sigma}*f_j  \ge \rho \phi_{\sigma}*f_0  \ge K f_0,
\end{eqnarray*}
where $K = \min\{\rho/3,\rho\,C\}$. Applying \eqref{eq:approximation}, the first integral on the r.h.s. of \eqref{integral} can be bounded by 
\begin{align*} 
\int_{A_\sigma} \frac{(f_0 - \phi_{\sigma}*f_j)^2}{\phi_{\sigma}* f_j} 
&\le \int_{A_{\sigma}} \frac{[f_0(x) - f_0(x)(1 + O(D(x)\sigma^{\beta}))]^2}{K f_0(x)} \\
&\precsim \int_{A_{\sigma}} f_0(x) O(D^2(x)\sigma^{2\beta}) = O(\sigma^{2\beta}) \nonumber.
\end{align*}
To bound the second integral of r.h.s in \eqref{integral}, according to Remark 3 in \cite{ghosal1999posterior} we get $\phi_{\sigma}*f_j \ge \rho f_0/3$, then we can find a constant $C < 1$ such that $\phi_{\sigma}*f_j \ge C f_0$. The second and third term in \eqref{integral} can be bounded by $O(\sigma^{2\beta})$ based on Lemma \ref{lem:tail}.
\end{proof}

\begin{lem}\label{lem:entropy}
Let $\mathbb{H}_1^a$ denote the unit ball of RKHS of the Gaussian process with rescaled parameter $a$ and $\mathbb{B}_1$ be the unit ball of $C[0,1]$. For $r >1$, there exists a constant $K$, such that for $\epsilon < 1/2$,
\begin{eqnarray}\label{eq:entropy}
\log N(\epsilon, \cup_{a \in [0,r]} \mathbb{H}_1^a, \norm{\cdot}_{\infty}) \le Kr\bigg( \log\frac{1}{\epsilon} \bigg)^2. 
\end{eqnarray}
\end{lem}

\begin{proof}
 Since we can write any element of $\mathbb{H}_1^a$ as a function of $\operatorname{Re}(z)$ by Lemma 4.5 in \cite{van2009adaptive}, and an $\epsilon$-net denoted by $\mathcal{F}^a$ over $\mathbb{H}_1^a$ is constructed through a finite set of piece-wise polynomial functions, and according to Lemma 4.4 and Lemma 4.5 in \cite{bhattacharya2014anisotropic}, 
$\mathcal{F}^a$ also forms an $\epsilon$-net over $\mathbb{H}_1^b$ as long as $a$ is sufficiently close to $b$. Thus we can find one set $\Gamma = \{a_i, i = 1, \dots, k\}$ with $k = \lfloor r\rfloor +1$ and $a_k = r$, such that for any $b\in [0,r]$ there exists some $a_i$ satisfying $|b-a_i|\le 1$, so that $\cup_{i\le k}\mathcal{F}^{a_i}$ forms an $\epsilon$-net over $\cup_{a \le r} \mathbb{H}_1^a$. Since the covering number of $\cup_{i\le k}\mathcal{F}^{a_i}$ is bounded by summation of covering number of $\mathcal{F}^{a_i}$, we obtain
\begin{eqnarray*}
\log N\big(\epsilon, \cup_{a \in [0,r]} \mathbb{H}_1^a, \|\cdot\|_{\infty}\big) \le \log \bigg(\sum_{i=1}^k \#(\mathcal{F}^{a_i})\bigg) \le \log(k \cdot \#(\mathcal{F}^r)) \le  Kr\bigg( \log\frac{1}{\epsilon} \bigg)^2.
 \end{eqnarray*} 
Here we write $\#(A)$ to denote the cardinality of any arbitrary set $A$. To prove the second inequality above, note that the piece-wise polynomials are constructed on the partition over $[0,1]$, denoted by $\cup_{i\le m}B_i$, where $B_i$'s are disjoint interval with length $R$ that can be considered as a non-increasing function of $a$,
so the total number of polynomials is non-decreasing in $a$. Also we find that when building the mesh grid of the coefficients of polynomials in each $B_i$, both the approximation error and tail estimate are invariant to interval length $R$, therefore we have $\#(\mathcal{F}^a) \le \#(\mathcal{F}^b)$ if $a \le b$, for $a,b \in [0,r]$.  
\end{proof}
\begin{remark}
With larger $a$ we need a finer partition on $[0,1]$ while the grid of coefficients of piece-wise polynomial remains the same except the range and the meshwidth will change together along with $a$. Since we can see the element $h$ of RKHS ball as a function of $it$ and with Cauchy formula we can bound the derivatives of $h$ by $C/R^n$, where $|h|^2 \le C^2$. 
\end{remark}

\bibliography{References, gpt28_1}
\end{document}